\documentclass[11pt, leqno]{article}

\usepackage{euscript,amstext,amsmath,amsthm,a4,amssymb}
\def\bC{{\rm \bf C}}
\def\bN{{\rm \bf N}}

\def\bR{{\rm \bf R}}

\def\bZ{{\rm \bf Z}}
\def\ch{{\rm  ch}}
\def\td{{\rm  Td}}
\def\rk{{\rm  rk}}
\def\End{{\rm  End}}
\def\ind{{\rm  Ind}}

\newcommand{\Ker}{\rm Ker}
\newcommand{\tr}{{\rm Tr}}
\newcommand{\rL}{{\rm L}}

\newcommand{\var}{\varepsilon}
\newcommand{\wi}{ \widetilde}

\newcommand{\comment}[1]{}

\newtheorem{thm}{Theorem}[section]

\newtheorem{prop}[thm]{Proposition}
 \newtheorem{cor}[thm]{Corollary}
\newtheorem{lemma}[thm]{Lemma}
\theoremstyle{remark}
\newtheorem{Rem}[thm]{Remark}
\theoremstyle{definition}
\newtheorem{defn}[thm]{Definition}

\title{Adiabatic limit, Bismut-Freed connection, and \\ the real analytic  torsion form}

\author{Xianzhe Dai$\ $  and$\ $  Weiping Zhang}
\date{}

\begin{document}

\maketitle
\comment{
\documentclass[11pt,reqno]{article}
\usepackage{mathrsfs}
\usepackage{amscd}
\usepackage{amssymb}
\usepackage{amsfonts}
\usepackage{euscript,amstext,amsthm,amssymb}
\usepackage{amsmath}

\textwidth 6in
\textheight 8.5in
\topmargin -0.25in
\oddsidemargin 0.25in

\newcommand{\be}{\begin{equation}}
      \newcommand{\ee}{\end{equation}}
      \newcommand{\ba}{\begin{eqnarray}}
       \newcommand{\ea}{\end{eqnarray}}
\newcommand{\ban}{\begin{eqnarray*}}
       \newcommand{\ean}{\end{eqnarray*}}

\newcommand{\field}[1]{\mathbb{#1}}
\newcommand{\bZ}{\field{Z}}
\newcommand{\bR}{\field{R}}
\newcommand{\bC}{\field{C}}
\newcommand{\bN}{\field{N}}
\def\cC{\mathscr {C}}
\def\cL{\mathscr {L}}
\def\cO{\mathscr {O}}
\def\cR{\mathscr {R}}
\def\mA{\mathcal{A}}
\def\mB{\mathcal{B}}
\def\mC{\mathcal{C}}
\def\mL{\mathcal{L}}
\def\mO{\mathcal{O}}
\def\mQ{\mathcal{Q}}
\def\mR{\mathcal{R}}
\def\bk{{\bf k}}
\def\bm{{\bf m}}
\def\br{{\bf r}}
\def\Re{{\rm Re}}
\def\Im{{\rm Im}}

\newcommand{\til}[1]{\widetilde{#1}}
\newcommand{\tx}{\til{X}}
\newcommand{\tj}{\til{J}}
\newcommand{\tg}{\til{g}}
\newcommand{\tsx}{\til{x}}
\newcommand{\ty}{\til{y}}
\newcommand{\tom}{\til{\omega}}
\newcommand{\tl}{\til{L}}
\newcommand{\te}{\til{E}}
\newcommand{\tdk}{\til{D}_k}
\newcommand{\tdel}{\til{\Delta}_k}
\newcommand{\tp}{\til{P}}
\newcommand{\tc}{\til{c}}
\newcommand{\tta}{\til{\tau}}
\newcommand{\tsc}{\til{\Delta}^\#_k}
\newcommand{\cali}[1]{\mathcal{#1}}
\newcommand{\Ha}{\cali{H}}

\newcommand{\etD}{e^{-tD^2}}
\newcommand{\etDu}{ e^{-tD_u^2}}
\newcommand{\bM}{\partial M}
\newcommand{\Db}{D_{\partial}}
\newcommand{\Dd}{\bar{D}}
\newcommand{\reta}{\overline{\eta}}
\newcommand{\sgn}{\mbox{\rm sgn}}
\newcommand{\TrbM}{\mbox{\rm Tr}^{\bM}}
\newcommand{\Tr}{\mbox{\rm Tr}}
\newcommand{\DET}{\mbox{\rm DET}}
\newcommand{\erfc}{\mbox{\rm erfc}}
\newcommand{\Dg}{g^{-1}[D, g]}
\newcommand{\Dbg}{ g^{-1}[\Db, g]}

\DeclareMathOperator{\End}{End} \DeclareMathOperator{\Hom}{Hom}
\DeclareMathOperator{\Ker}{Ker} \DeclareMathOperator{\Dom}{Dom}
\DeclareMathOperator{\ran}{Range} \DeclareMathOperator{\rank}{rank}
\DeclareMathOperator{\Id}{Id} \DeclareMathOperator{\supp}{supp}
\DeclareMathOperator{\tr}{Tr} \DeclareMathOperator{\ind}{index}
\DeclareMathOperator{\td}{Td} \DeclareMathOperator{\vol}{vol}
\DeclareMathOperator{\ch}{ch}
\newcommand{\spin}{$\text{spin}^c$ }
\newcommand{\spec}{\rm Spec}

\newcommand{\norm}[1]{\lVert#1\rVert}
\newcommand{\abs}[1]{\lvert#1\rvert}
\newcommand{\om}{\omega}

\newtheorem{thm}{Theorem}[section]
\newtheorem{lemma}[thm]{Lemma}
\newtheorem{prop}[thm]{Proposition}
\newtheorem{cor}[thm]{Corollary}
\theoremstyle{definition}
\newtheorem{defn}[thm]{Definition}
\theoremstyle{remark}
\newtheorem{rem}[thm]{Remark}
\newcommand{\wi}{\widetilde}
\newcommand{\var}{\varepsilon}

\newcommand{\comment}[1]{}

\begin{document}
\title{Adiabatic limit, Bismut-Freed connection, and real torsion form}
\maketitle

\begin{center}
\author{Xianzhe Dai$\ $  and$\ $  Weiping Zhang}
\end{center}



\comment}

\begin{abstract}For a complex flat vector bundle over a fibered
manifold, we consider the 1-parameter family of certain deformed
sub-signature operators introduced by Ma-Zhang in \cite{MZ}. We
compute the  adiabatic limit of the  Bismut-Freed connection
associated to this family and show that the Bismut-Lott analytic
torsion form shows up naturally under this procedure.
\end{abstract}

\renewcommand{\theequation}{\thesection.\arabic{equation}}
\setcounter{equation}{0}

\section{Introduction} \label{s1}

$\quad$ Adiabatic limit refers to the geometric degeneration when
metric in certain directions are blown up, while the remaining
directions are kept fixed.

Typically, the underlying manifold has a so called fibration
structure (or fiber bundle structure). That is
\[ F \longrightarrow M \stackrel{\pi}{\longrightarrow} B, \]
where $\pi$ is a submersion and $F \simeq F_b=\pi^{-1}(b)$, for
$b\in B$. Given a submersion metric on $M$:
\[ g=\pi^*g_B + g_F, \]
the adiabatic limit refers to the limit as $\epsilon \rightarrow 0$
of
\[ g_{\epsilon}=\epsilon^{-2}\pi^*g_B + g_F. \]
This is first introduced by Witten \cite{w} in his famous work on global
gravitational anomalies.

Witten considered the adiabatic limit of the eta invariant of
Atiyah-Patodi-Singer\cite{APS1}-\cite{APS3}. Full mathematical
treatment and generalizations are given by Bismut-Freed \cite{BF},
Cheeger \cite{C}, Bismut-Cheeger \cite{BC}, Dai \cite{D} among
others. The adiabatic limit of the eta invariant gives rise to the
Bismut-Cheeger eta form, a canonically defined differential form on
the base $B$. The eta form is a higher dimensional generalization of
the eta invariant as it gives the boundary contribution of the
family index theorem for manifolds with boundaries, see
Bismut-Cheeger \cite{bc1,bc2}, and Melrose-Piazza \cite{mp1,mp2}.
The degree zero component of the eta form here is exactly the eta
invariants of the fibers. The nonzero degree components therefore
contains new geometric information about the fibration.

Another important geometric invariant is the analytic torsion. The
adiabatic limit of the analytic torsion has been considered by
Dai-Melrose \cite{dm} (see also the topological treatment of Fried
\cite{fri}, Freed \cite{fr}, and L\"uck-Schick-Thielmann
\cite{lst}). In contrast to the case of the eta invariant, the
adiabatic limit here does not give rise to a higher invariant. This
is because the associated characteristic class involved here is the
Pfaffian, a top form which kills any possible higher degree
components arising from the adiabatic limit.

It should be noted that there is a complex analogue of the analytic
torsion for complex manifolds called the holomorphic torsion. Its
adiabatic limit has been considered by Berthomieu-Bismut
\cite{BerB}. And it does produce the holomorphic torsion form of
Bismut-K\"ohler \cite{bk}. The difference can be explained by the
fact that the characteristic class here is the Todd class---a stable
class.

There is another way to view the higher invariants, namely via
transgression. The eta form transgresses between the Chern-Weil
representative of the family index and its Atiyah-Singer
representative. Similarly, the holomorphic torsion form is the
double transgression of the family index in the complex setting.
Bismut-Lott \cite{BL} uses this view point to define the real
analytic torsion form, a higher dimensional generalization of the
analytic torsion. It is a canonical transgression of certain odd
cohomology classes.

There remains the question of whether the real analytic torsion form
can be obtain from the adiabatic limit process. The purpose of this
paper is to answer this question in the affirmative. We show that,
if one considers the Bismut-Freed connection of the 1-parameter
family of certain deformed sub-signature operators introduced by
Ma-Zhang in \cite{MZ}, its adiabatic limit essentially gives rise to
the Bismut-Lott real analytic torsion form. In fact, it is precisely
the positive degree component of the real analytic torsion form that
is captured here. This should be compared with \cite{dm} where the
adiabatic limit of the analytic torsion captures only the degree $0$
part of the real analytic torsion form.

The paper is organized as follows. We first look at the finite
dimensional case in Section 2. Thus in \S \ref{s2.1}, we introduce
flat cochain complexes, flat superconnections and their
rescalings. The family of deformed sub-signature operators is then
introduced in \S \ref{s2.2}. After some preparatory results, we
define an invariant which should be interpreted as the imaginary
part of the Bismut-Freed connection form for the family of the
deformed sub-signature operators. Finally in \S \ref{s2.3}, we
study the adiabatic limit of our invariant. The fibration case is
set up as an infinite dimensional analog and studied in Section 3.
The flat superconnection in this case is the Bismut-Lott
superconnection and is recalled  in \S 3.1. In \S 3.2, we discuss
the analog of the deformed sub-signature operators in the
fibration case. Then we look into the Bismut-Freed connection and
define a corresponding invariant in \S 3.3. Finally, we study the
adiabatic limit of our invariant in \S 3.4. Our main result is
stated in Theorem \ref{mt}. In \S 3.5, we compare the Bismut-Lott
real analytic torsion form with the torsion form coming out of the
adiabatic limit.

\section{The finite dimensional  case}\label{s2}

$\quad$ In this section, we study the finite dimensional case where
instead of a flat vector bundle over a fibered manifold, we consider
the situation of a flat cochain complex over an even dimensional
manifold. This fits well with the structures considered in \cite{BL}
and \cite{MZ}. The fibered manifold case is the infinite dimensional
analog which will be considered in the next section.

\subsection{Supperconnections and flat cochain complex} \label{s2.1}

$\quad$ Let $(E, v)$ be a $\mathbb{Z}$-graded cochain complex of
 finite rank complex vector bundles over a closed manifold $B$,
\begin{align} \label{2.1}
(E,v):\ \ 0 \rightarrow E^0 \stackrel{v}{\rightarrow} E^1
\stackrel{v}{\rightarrow} \cdots \stackrel{v}{\rightarrow} E^n
\rightarrow 0.
\end{align}
Let $\nabla^E=\oplus_{i=0}^n \nabla^{E^i}$ be a
$\mathbb{Z}$-graded connection on $E$. We call $(E, v,\nabla^E)$ a
flat cochain complex if the following two conditions hold,
\begin{align}\label{2.2}
\left(\nabla^E\right)^2=0, \ \ \ \ \ \ \left[ \nabla^E, \
v\right]=0.
\end{align}

Let $h^E=\oplus_{i=0}^n h^{E^i}$ be a $\mathbb{Z}$-graded
hermitian metric on $E$ and denote by $v^*:\ E^* \rightarrow
E^{*-1}$ the adjoint of $v$ with respect to $h^E$. Let
$(\nabla^E)^*$ denote the adjoint connection of $\nabla^E$ with
respect to $g^E$. Then (cf. \cite[(4.1), (4.2)]{BZ} and \cite[\S
1(g)]{BL})
\begin{align}\label{2.3} \left(\nabla^E\right)^*=\nabla^E + \omega\left(E, h^E\right), \end{align}
where
 \begin{align}\label{2.4}
 \omega(E,
h^E)=\left(h^E\right)^{-1}\left(\nabla^Eh^E\right) .
\end{align}

Consider the superconnections on $E$ in the sense of Quillen
\cite{Q} defined by
\begin{align}\label{2.5}
 A'=\nabla^E+v,\ \ \ \ \ \ A''=\left(\nabla^E\right)^*+v^*.
  \end{align}

Let $N\in\End (E)$ denote the number operator of $E$ which acts on
$E^i$ by multiplication by $i$. We extend $N$ to an element of
$\Omega^0(B, \End( E))$.

Following \cite[(2.26), (2.30)]{BL}, for any $u>0$, set
\begin{align}\label{2.6}
&C_u'=u^{N/2}A'u^{-N/2}=\nabla^E+\sqrt{u}v,\\
&C_u''=u^{-N/2}A''u^{N/2}=\left(\nabla^E\right)^*+\sqrt{u}v^*,\nonumber\\
&C_u={1\over 2}\left(C_u'+C_u''\right),\ \ \ \ \ \ \ D_u={1\over
2}\left(C_u''-C_u'\right). \nonumber
\end{align}
Then we have
\begin{align}\label{2.61}
C_u^2=-D_u^2,\ \ \ \ \ \ \ \left[C_u,D_u\right]=0.
\end{align}
 Let
\begin{align}\label{2.7}
\nabla^{E,e}=\nabla^E+{1\over 2}\omega \left(E,h^E\right)
\end{align}
be the Hermitian connection on $(E, h^E)$ (cf. \cite[(1.33)]{BL}
and \cite[(4.3)]{BZ}). Then
\begin{align}\label{2.8}
C_u=\nabla^{E,e}+{\sqrt{u}\over 2}\left(v+v^*\right)
\end{align}
is a superconnection on $E$, while
\begin{align}\label{2.9}
D_u={1\over 2}\omega \left(E,h^E\right)+{\sqrt{u}\over
2}\left(v^*-v\right)
\end{align}
is an odd element in $C^\infty(B, \Lambda^*(B)\widehat{\otimes}
{\rm End}(E)) $.

\subsection{Deformed signature operators and the Bismut-Freed
connection}\label{s2.2}

$\quad$ We assume in the rest of this section that $p=\dim B$ is
even and $B$ is oriented.

 Let $g^{TB}$ be a Riemannian metric on $TB$. For $X\in TB$, let $c(X)$, $\widehat{c}(X)$
 be the Clifford actions on $\Lambda(T^*B)$ defined by $c(X)=X^*-i_X$, $\widehat{c}(X)
 =X^*+i_X$, where $X^*\in T^*B$ corresponds to $X$ via $g^{TB}$ (cf. \cite[(3.18)]{BL} and \cite[\S 4(d)]{BZ}).
 Then for any $X$, $Y\in TB$,
 \begin{align}\label{2.10}
& c(X)c(Y)+c(Y)c(X)=-2\langle X,Y\rangle, \\
& \widehat{c}(X)\widehat{c}(Y)+\widehat{c}(Y)\widehat{c}(X)
=2\langle X,Y\rangle,  \nonumber \\
& c(X)\widehat{c}(Y)+\widehat{c}(Y)c(X)=0.\nonumber
\end{align}

Let $e_1, \cdots, e_p$ be a (local) oriented  orthonormal basis of
$TB$.
  Set
 \begin{align}\label{2.11}
\tau=\left(\sqrt{-1}\right)^{{p(p+1)\over 2}}c(e_1)\cdots c(e_p) .
\end{align}
 Then $\tau$ is a well-defined self-adjoint  element such that
 \begin{align}\label{2.12}
\tau^2={\rm Id}|_{\Lambda(T^*B)}.
\end{align}

 Let $\mu$ be a Hermitian vector bundle on $B$ carrying a
 Hermitian connection $\nabla^\mu$ with the curvature denoted by
$R ^\mu = (\nabla^{\mu})^2$. Let  $\nabla^{TB}$ be the  Levi-Civita
connection on $(TB, g^{TB})$ with its curvature $R^{TB}$. Let
$\nabla^{\Lambda(T^*B)}$ be the Hermitian connection on
$\Lambda(T^*B)$
 canonically induced from $\nabla^{TB}$.
 Let $\nabla^{\Lambda(T^*B)\otimes\mu\otimes E,e}$ be the tensor product
 connection on $\Lambda(T^*B)\otimes\mu\otimes E$ given by
 \begin{align}\label{2.13}
\nabla^{\Lambda(T^*B)\otimes\mu\otimes E,e}=
 \nabla^{\Lambda(T^*B)}\otimes {\rm Id}_{\mu\otimes E}
 +{\rm Id}_{\Lambda(T^*B)}\otimes\nabla^\mu\otimes{\rm Id}_E
 +{\rm Id}_{\Lambda(T^*B)\otimes\mu}\otimes\nabla^{E,e}.
\end{align}

 Let the Clifford actions $c$, $\widehat{c}$ extend to actions
 on $\Lambda(T^*B)\otimes\mu\otimes E$ by acting as identity on $\mu\otimes E$.
Let $\varepsilon$ be the induced ${\bf Z}_2$-grading operator on
$E$, i.e.,
 $\varepsilon=(-1)^N$ on $E$. We extend $\varepsilon$ to an action on
 $\Lambda(T^*B)\otimes\mu\otimes E$ by acting as identity on $\Lambda(T^*B)\otimes\mu$.

 Let $\tau\widehat{\otimes}\varepsilon$ define the
 $\mathbb{Z}_2$-grading on $(\Lambda(T^*B)\otimes\mu)\widehat{\otimes}
 E$, then
\begin{align}\label{2.14} D_{\rm sig}^{\mu\otimes
E}= \sum_{i=1}^pc(e_i)
 \nabla_{e_i}^{\Lambda^*(T^*B)\otimes\mu\otimes E,e} \end{align}
defines the twisted signature operator with respect to this
$\mathbb{Z}_2$-grading. Playing an important role here is its
deformation, given by
\begin{align}\label{2.15}
 D_{{\rm sig},u}^{\mu\otimes E}=D_{\rm sig}^{\mu\otimes E}
+{\sqrt{u}\over 2}\left(v+v^*\right),  \end{align} with $u> 0$,
which might be thought of as a quantization of $C_u$.

   Let $Y_u$ be the   skew adjoint
  element in $\End^{\rm odd} (\Lambda^*(T^*B)\otimes\mu\otimes E)$ defined
  by (cf.  \cite[(2.18)]{MZ})
\begin{align}\label{2.16}
Y_u={1\over 2}\sum_{i=1}^pc(e_i)\omega\left(E,h^E\right)(e_i)
+{\sqrt{u}\over 2}\left(v^*-v\right),
\end{align}
which might be thought of as a quantization of $D_u$.

Now following \cite[Definition 2.3]{MZ}, for any $r\in {\bf R}$,
define
\begin{align}\label{2.17}
 D_{{\rm sig},u}^{\mu\otimes E}(r)=D_{{\rm sig},u}^{\mu\otimes E}+
\sqrt{-1}r Y_u.
\end{align}

From (\ref{2.14})-(\ref{2.17}), one has (cf. \cite[(2.22)]{MZ})
\begin{multline} \label{2.18}
 D_{{\rm sig},u}^{\mu\otimes E}(r)=\sum_{i=1}^p{c}(e_i)
 \left(\nabla_{e_i}^{\Lambda^*(T^*B)\otimes\mu\otimes E,e}
 +{\sqrt{-1}r\over 2}
 \omega\left(E,h^E\right)(e_i)\right) \\
+{\sqrt{u}\over
2}\left(\left(1-\sqrt{-1}r\right)v+\left(1+\sqrt{-1}r\right)v^*\right).
\end{multline}

\begin{prop}\label{axpd} We have the following asymptotic
expansion
\begin{multline}\label{eaxpd}
 {\rm Tr}_s\left[  D^{\mu\otimes
E}_{{\rm sig},u}(r) Y_u e^{-t \left(D^{\mu\otimes E}_{{\rm
sig},u}(r)\right)^2}\right]
  =  c_0(u,r)\,  + c_{1}(u,r)\, t + \cdots \end{multline}
as $t \rightarrow 0$. The expansion is uniform for $(u, r)$ in a
compact set.
\end{prop}
\begin{proof}
We introduce two auxiliary Grassmann variables $z_1, z_2$ and write
\begin{align}\label{auxgr} {\rm Tr}_s\left[  D^{\mu\otimes
E}_{{\rm sig},u}(r)Y_u e^{-t \left(D^{\mu\otimes E}_{{\rm
sig},u}(r)\right)^2}\right] = \hspace{2in} \\
- t^{-2}{\rm Tr}_{s, z_1, z_2}\left[ e^{-t \left([D^{\mu\otimes
E}_{{\rm sig},u}(r)]^2 -z_1 D^{\mu\otimes E}_{{\rm sig},u}(r)- z_2
Y_u \right)} \right] \nonumber .
\end{align}
Here the minus sign comes from the order of the appearance of $z_1,
z_2$ and $D^{\mu\otimes E}_{{\rm sig},u}(r), Y_u$.

Applying the standard elliptic theory to the right hand side of
(\ref{auxgr}), we derive an asymptotic expansion
\begin{multline}
{\rm Tr}_s\left[ D^{\mu\otimes E}_{{\rm sig},u}(r) Y_u  e^{-t
\left(D^{\mu\otimes E}_{{\rm sig},u}(r)\right)^2}\right]= c_{-p/2
-2}(u,r) t^{-p/2-2} + c_{-p/2 -1}(u,r) t^{-p/2-1} + \nonumber \\
\hspace{2.5in} \cdots +c_{0}(u,r) + c_{1}(u,r)\,
t + \cdots . \nonumber
\end{multline}
On the other hand, by the Lichnerowicz formula (Cf. (\ref{2.27}))
and the same argument as in \cite{BF}, we have
\[  \lim_{t\rightarrow 0} t\, {\rm Tr}_s\left[ D^{\mu\otimes E}_{{\rm sig},u}(r) Y_u  e^{-t
\left(D^{\mu\otimes E}_{{\rm sig},u}(r)\right)^2}\right]=0.
\]
It follows then that $c_{i}(u,r)=0$ for $-p/2 -2\leq i \leq -1$.
Thus, the asymptotic expansion starts with the constant term.
\end{proof}

Since
\[ {\rm Tr}_s\left[ D^{\mu\otimes E}_{{\rm sig},u}(r) Y_u  e^{-t
\left(D^{\mu\otimes E}_{{\rm sig},u}(r)\right)^2}\right] =
- {\rm Tr}_s\left[  Y_u  D^{\mu\otimes E}_{{\rm sig},u}(r) e^{-t
\left(D^{\mu\otimes E}_{{\rm sig},u}(r)\right)^2}\right] \]
is also exponentially decaying as $t\rightarrow \infty$, the quantity on the right hand side of the following
definition is well-defined.

 \begin{defn}\label{t2.1} We define
\begin{align} \label{2.19}
\delta_u(E, v)(r) = \int_0^{\infty} {\rm Tr}_s\left[ D^{\mu\otimes
E}_{{\rm sig},u}(r)Y_u e^{-t \left(D^{\mu\otimes E}_{{\rm
sig},u}(r)\right)^2}\right] dt.
\end{align}
\end{defn}

\begin{Rem}\label{t2.2} If $H^*(E,v)=\{0\}$, i.e., $(E,v)$ is
acyclic,   by \cite[(2.27)]{MZ}, which we recall as follows,
\begin{align} \label{2.20}
\left(\left(1-\sqrt{-1}r\right)v+\left(1+\sqrt{-1}r\right)v^*\right)^2=\left(1+r^2\right)
\left(v+v^*\right)^2,
\end{align}
 (\ref{2.18}) and  proceed as in \cite{BC}, one sees that when $u>0$ is large
enough, $D_{{\rm sig},u}^{\mu\otimes E}(r)$ is invertible for
fixed $r\in {\bf R}$. Since by (\ref{2.17}) one has
$${\partial D^{\mu\otimes E}_{{\rm
sig},u}(r)\over\partial r}=\sqrt{-1}Y_u,$$
${\sqrt{-1}\over 2}\delta_u(E, v)(r)$ is (the imaginary part of) the
Bismut-Freed connection form (\cite{BF}, see also \cite[(3.8)]{DF})  over $r$ of the Quillen
determinant line bundle of the $r$-family operators $\{D_{{\rm
sig},u}^{\mu\otimes E}(r)\}_{r\in{\mathbb R}}$.
\end{Rem}

\subsection{Adiabatic limit as $u\rightarrow +\infty$}\label{s2.3}

$\quad$ We first rewrite $\delta_u(E, v)(r)$ as
\begin{align} \label{2.21} \delta_u(E, v)(r) =
\int_0^{\infty} {\rm Tr}_s\left[ D_{\epsilon}(r)Y^{\epsilon} e^{-t
D^2_{\epsilon}(r)}\right] dt,
\end{align}
where $\epsilon=u^{-\frac{1}{2}}$ and
\begin{align} \label{2.22} Y^{\epsilon}=\frac{\epsilon }{2} c(\omega) + \frac{1}{2}(v^*-v), \end{align}
\[ D_{\epsilon}(r)=\epsilon D_{{\rm sig},u}^{\mu\otimes E}+
\sqrt{-1}r Y^{\epsilon}. \]

We fix a square root of $\sqrt{-1}$ and let
$\varphi:\Lambda(T^*B)\rightarrow\Lambda(T^*B)$ be the
homomorphism defined by
$\varphi:\omega\in\Lambda^i(T^*B)\rightarrow (2\pi
\sqrt{-1})^{-i/2}\omega.$ The formulas in what follows will not
depend on the choice of the square root of $\sqrt{-1}$.

Let $\rL(TB,\nabla^{TB})$ be the Hirzebruch characteristic form
defined by
 $$\rL(TB,\nabla^{TB})=\varphi\,
 {\det}^{1/2}\left({R^{TB}\over \tanh\left({R^{TB}/2}\right)} \right),$$
 while ${\rm ch}(\mu,\nabla^\mu)$ be the Chern character form
 defined by
 $$
{\rm ch}(\mu,\nabla^\mu)=\varphi\,\tr\left[\exp(-R^\mu)\right].$$

\begin{prop}\label{t2.3} We have
\begin{multline}\label{2.23}
 \lim_{\epsilon\rightarrow 0} {\rm Tr}_s\left[ D_{\epsilon}(r)Y^{\epsilon} e^{-t
 D^2_{\epsilon}(r)}\right]
  =  - 
  \int_B L\left(TB, \nabla^{TB}\right)\ch\left(\mu,
  \nabla^{\mu}\right)\\ \cdot
 \varphi\,{\rm Tr}_s\left[t^{- {1\over 2} }D_t
 \left(
\frac{1}{2}\left(v+v^*\right) + {\sqrt{-1}r\over
2}\left(v^*-v\right)\right)
   e^{- \left(C_t+\sqrt{-1} r D_t\right)^2} \right]  .\end{multline}
\end{prop}
\begin{proof}
 As in \cite{BF} and \cite{BC}, we introduce an auxiliary Grassmann variable $z$ and rewrite
\begin{align}\label{2.24} {\rm Tr}_s\left[ D_{\epsilon}(r)Y^{\epsilon} e^{-t D^2_{\epsilon}(r)}\right] =
-{\rm Tr}_{s, z}\left[Y^{\epsilon} t^{-\frac{1}{2}} e^{-t
D^2_{\epsilon}(r)+ z\sqrt{t} D_{\epsilon}(r)}\right],
 \end{align}
 where for elements of the form $A+zB$ with $A,\ B$ containing no
 $z$, we have as in \cite{BF} and \cite{BC} that ${\rm
 Tr}_{s,z}[A+zB]={\rm
 Tr}_{s}[B].$

By (\ref{2.15}) and (\ref{2.22}), one has
\begin{multline}\label{2.25}
D_{\epsilon}(r)=\epsilon D_{{\rm sig}}^{\mu\otimes E} + \frac{1
}{2}\left(v+ v^*\right) + \sqrt{-1}r \left(\frac{\epsilon }{2}
c(\omega) + \frac{1}{2}\left(v^*-v\right)\right)\\
=\epsilon\left( D_{{\rm sig}}^{\mu\otimes E}+ \sqrt{-1}r
 \frac{1 }{2} c(\omega)\right)+\frac{1 }{2} \left(v+ v^*\right) + \sqrt{-1}r
\frac{1}{2}\left(v^*-v\right).
\end{multline}
Denote by
\begin{align}\label{2.26}
V=\frac{1 }{2} \left(v+ v^*\right) + \sqrt{-1}r
\frac{1}{2}\left(v^*-v\right).
\end{align}
By Lichnerowicz formula, we have (for simplicity we denote
$\nabla=\nabla^{\Lambda^*(T^*B)\otimes\mu\otimes E,e}$)
\begin{multline}\label{2.27}
t D^2_{\epsilon}(r)-z\sqrt{t} D_{\epsilon}(r) =  t\epsilon^2
\left(D_{{\rm sig}}^{\mu\otimes E}\right)^2 + t\epsilon^2 \left[
D_{{\rm sig}}^{\mu\otimes E},
\sqrt{-1}r \frac{1}{2} c(\omega)\right] + t\epsilon \left[ D_{{\rm sig}}^{\mu\otimes E},  V\right] \\
 + t \left( \sqrt{-1}r \frac{\epsilon }{2} c(\omega) + V\right)^2
 - z\sqrt{t}\left(\epsilon D_{{\rm sig}}^{\mu\otimes E} + \sqrt{-1}r \frac{\epsilon }{2} c(\omega)
+ V\right) \\
 =  -t\left(\epsilon \nabla_{e_i} + \frac{1}{2\sqrt{t}}z  {c}(e_i)\right)^2
 +\frac{t\epsilon^2}{4} k^{TB}+ \frac{t\epsilon^2}{2}  {c}(e_i)  {c}(e_j)\otimes R^{\mu\otimes E, e}(e_i, e_j)
\\
 + \frac{t\epsilon^2}{8} R^{TB}_{ijkl} c(e_i)c(e_j) \hat{c}(e_k)
 \hat{c}(e_l) + t\epsilon c(e_i)\nabla_{e_i}V + \frac{\sqrt{-1} r}{2} t\epsilon^2
c(e_i)c(e_j) \nabla_{e_i}\omega_j \\
-t\epsilon^2\sqrt{-1}r\sum_{i=1}^p \omega(e_i)\nabla_{e_i}
 + t\left(\frac{\sqrt{-1}r}{2} \epsilon c(\omega) + V\right)^2 - z\sqrt{t}  V -
  z \frac{\sqrt{-1}r}{2} \sqrt{t} \epsilon c(\omega) ,
\end{multline}
where $R^{TB}$ is the Riemannian curvature and $k^{TB}$ is the
scalar curvature of $g^{TB}$, while $R^{\mu\otimes E, e}$ is the
curvature of the connection on $\mu\otimes E$ obtained through
$\nabla^\mu$ and $\nabla^{E,e}$.

Now we find ourself exactly in the situation of \cite{BC}. Near
any point $x$, take a normal coordinate system $\{x_i\}$ and the
associated orthonormal basis $\{e_i\}$.  We first conjugate $t
D^2_{\epsilon}(r)-z\sqrt{t} D_{\epsilon}(r)$ by the exponential
$e^{\frac{z\sum_{i=1}^px_i c(e_i)}{2\sqrt{t}\epsilon}}$ and then
apply the Getzler transformation $G_{\sqrt{t}\epsilon}$. One finds
that after these procedures, the operator $t
D^2_{\epsilon}(r)-z\sqrt{t} D_{\epsilon}(r) $ tends to, as
$\epsilon\rightarrow 0$,
\begin{multline}\label{2.28}
  -\left(\partial_i + \frac{1}{4} R^{TB}_{ij} x_j\right)^2 + \frac{1}{4}R^{TB}_{kl}  \hat{c}(e_k)
 \hat{c}(e_l) + R^{\mu\otimes E, e}
+ t^{\frac{1}{2}} \nabla V + \frac{\sqrt{-1} r}{2} \nabla \omega \\
 + \left( \frac{\sqrt{-1} r}{2} \omega + t^{\frac{1}{2}} V\right)^2 - z\sqrt{t}
  V  \\
 =   -\left(\partial_i + \frac{1}{4} R^{TB}_{ij} x_j\right)^2 + \frac{1}{4}R^{TB}_{kl}
  \hat{c}(e_k) \hat{c}(e_l) +R^\mu\\ +
\left(\nabla^{E,e} + \frac{\sqrt{-1} r}{2} \omega +
t^{\frac{1}{2}} V\right)^2
-z  \sqrt{t}  V   \\
 =  -\left(\partial_i + \frac{1}{4} R^{TB}_{ij} x_j\right)^2 + \frac{1}{4}R^{TB}_{kl}  \hat{c}(e_k) \hat{c}(e_l) +
\left(C_t+\sqrt{-1} r D_t\right)^2 \\ - z\sqrt{t}\left(
\frac{1}{2}\left(v+v^*\right) + {\sqrt{-1}r\over
2}\left(v^*-v\right)\right) .
\end{multline}

On the other hand, by (\ref{2.22}) it is clear that under the same
procedures,
   $Y^{\epsilon}$ tends to, as
$\epsilon\rightarrow 0$,
\begin{align}\label{2.29}
t^{-\frac{1}{2}}\left(\frac{1 }{2} \omega  +
\frac{\sqrt{t}}{2}\left(v^*-v\right)\right)=t^{-\frac{1}{2}}
D_t.\end{align}

From (\ref{2.24}), (\ref{2.28}) and (\ref{2.29}), by proceeding the
by now standard local index techniques, and keeping in mind that the
supertrace in the left hand sides of (\ref{2.23}) and (\ref{2.24})
are respect to the $\mathbb{Z}_2$-grading defined by
$\tau\widehat{\otimes}\varepsilon$, one derives (\ref{2.23}).
\end{proof}

We now examine the terms appearing in the right hand side of
(\ref{2.23}).

By (cf. \cite[(2.34)]{MZ})
\begin{align}\label{2.30}
\left(C_t+\sqrt{-1}r D_t\right)^2=\left(1+r^2\right)C_t^2
=-\left(1+r^2\right)D_t^2,
\end{align}
and
\begin{align} \label{2.33}v+v^* =
-2t^{-\frac{1}{2}} \left[N, D_t\right],\ \ \ \
v^*-v=-2t^{-\frac{1}{2}} \left[N, C_t\right],\end{align}
 we have
\begin{multline} \label{2.31}
-{\rm Tr}_s\left[t^{-{1\over 2}}D_t\left(
\frac{1}{2}\left(v+v^*\right) + {\sqrt{-1}r\over
2}\left(v^*-v\right)\right)
    e^{- \left(C_t+\sqrt{-1} r D_t\right)^2} \right]  \\
 = -\frac{ {1}}{2\sqrt{t}} {\rm Tr}_s\left[D_t\left(v+v^*\right)e^{\left(1+r^2\right)D_t^2}\right]
  - \frac{\sqrt{-1}r}{2\sqrt{t}} {\rm Tr}_s\left[D_t\left(v^*-v\right)e^{\left(1+r^2\right)D_t^2}\right]\\
=
  {1\over t}{\rm Tr}_s\left[D_t\left[N, D_t\right]e^{\left(1+r^2\right)D_t^2}\right]+
 {\sqrt{-1}r\over t} {\rm Tr}_s\left[D_t\left[N, C_t\right]e^{\left(1+r^2\right)D_t^2}\right]\\
  =   -\frac{ {1}}{t} {\rm Tr}_s\left[ND_t^2e^{\left(1+r^2\right)D_t^2}
  -D_tND_te^{\left(1+r^2\right)D_t^2}\right]
  +{\sqrt{-1}r\over t}d {\rm Tr}_s\left[
   N D_te^{\left(1+r^2\right)D_t^2}\right],
\end{multline}
where in the last equality we have used (\ref{2.61}) (compare also
with \cite[(2.75)]{MZ}).

We are now ready to prove the following main result of
this section.

\begin{thm}\label{t2.4} Under the assumption that the flat cochain complex $(E, v)$ is acyclic: $H^*(E,
v)=0$,
  the following identity holds,
\begin{align}\label{2.35}
{1\over 2}\lim_{u \rightarrow +\infty} \delta_u(E, v)(r)= \int_B
L\left(TB, \nabla^{TB}\right)
 \ch\left(\mu, \nabla^{\mu}\right) {\cal T}_r, \end{align}
where
\begin{align}\label{2.36}  {\cal T}_r=
 -\int_{0}^{\infty} \varphi\,
 {\rm Tr}_s\left[ND_t^2e^{\left(1+r^2\right)D_t^2}\right] {dt\over t}.  \end{align}
\end{thm}
\begin{proof}  First of all, the assumption that $H^*(E, v)=0$ implies that the eigenvalues of
$ D_{\epsilon}(r) $ are uniformly bounded away from zero. Hence the
integral in (\ref{2.21}) is uniformly convergent at $t=\infty$.

We now examine the same issue at $t=0$. From Proposition \ref{axpd}, one has
\begin{multline}
 {\rm Tr}_s\left[  D_{\epsilon}(r)   Y^{\epsilon} e^{-t \left( D_{\epsilon}(r) \right)^2}\right]
  =  c_{0}(\epsilon,r) + c_{1}(\epsilon,r)\, t  + \cdots .\end{multline}
  We claim that this asymptotic expansion is in fact uniform in $\epsilon$ as $\epsilon \rightarrow 0$ and the
  coefficients converge to that of asymptotic expansion of the right hand side of (\ref{2.23}).
  This can be seen by an argument similar to that of \cite{BC}, which is carried out in detail
  later for the infinite dimensional case; see the proof of Proposition \ref{t3.5}.

Our theorem now follows from Proposition \ref{t2.3}, the equation  (\ref{2.31}) and the above discussion.
\end{proof}

\begin{Rem}\label{t2.5} By Remark \ref{t2.2}, one sees that under
the assumption of Theorem \ref{t2.4}, for each $r\in {\mathbb R}$,
when $u>0$ is large enough, ${\sqrt{-1}\over 2}\delta_u(E, v)(r)$ is
the Bismut-Freed connection form of the $r$-family of operators
$D_{{\rm sig},u}^{\mu\otimes E}(r)$ at $r$. While on the other hand,
by comparing the right hand side of (\ref{2.36}) with \cite{BL} and
\cite{MZ}, one sees that ${\cal T}_r$ here gives, up to rescaling,
the nonzero degree terms of the Bismut-Lott torsion form
(\cite{BL}). Thus, we can say that one obtains the Bismut-Lott
torsion form through the adiabatic limit of the Bismut-Freed
connection. This is the main philosophy we would like to indicate in
this paper.
\end{Rem}

\section{Sub-signature operators, adiabatic limit and the Bismut-Lott torsion
form}\label{s3} \setcounter{equation}{0}

$\quad$ In this section, we deal with the fibration case. We will
show that, for an acyclic flat complex vector bundle over a fibered
manifold, if we consider the Bismut-Freed connection form \cite{BF}
on the Quillen determinant line bundle associated to the 1-parameter
family constructed  in \cite[Section 3]{MZ}, then the Bismut-Lott
analytic torsion form \cite{BL} will show up naturally through the
adiabatic limit of this connection form. This tautologically answers
a question asked implicitly in the original article of Bismut-Lott.

\subsection{The Bismut-Lott Superconnection}\label{s3.1}

$\quad$ We first set up the fibration case as an infinite
dimensional analog of the case considered in the previous section.
Let $\pi: M\to B$ be a smooth fiber bundle with compact fiber $Z$ of
dimension $n$. We denote by $m=\dim M,\ p=\dim B$. Let $TZ$ be the
vertical tangent bundle of the fiber bundle, and let $T^* Z$ be its
dual bundle.

 Let ${TM}=T^HM\oplus TZ$ be a splitting of $TM$.
\comment{ Let $T^H M$ be a sub-bundle of $TM$ such that
\begin{eqnarray}\label{a05}
TM = T^H M \oplus TZ.
\end{eqnarray} }
Let $P^{TZ}, P^{T^HM}$ denote the projection from $TM$ to $TZ,
T^HM$. If $U\in T B$, let $U^H$ be the lift of $U$
 in $T^H M$, so that
$\pi_* U^H = U$.

 Let $F$ be a flat complex vector bundle on $M$ and let $\nabla ^F$
denote its flat connection.

Let $E= \oplus_{i=0}^n E^i$ be the smooth infinite-dimensional
 ${\mathbb Z}$-graded vector bundle over $B$ whose fiber
over $b\in B$ is $C^{\infty}(Z_b, (\Lambda ( T^* Z)\otimes
F)_{|Z_b})$. That is
\begin{align}
C^{\infty}(B, E^i)= C^{\infty}(M, \Lambda^i ( T^* Z)\otimes F).
 \end{align}

\begin{defn} \label{t3.1}
For $s\in C^{\infty}(B, E)$ and $U$ a vector field on $B$, let
$\nabla^E$ be a ${\mathbb Z}$-grading  preserving connection on $E$
defined by
\begin{align}\label{a04}
\nabla^E_U s = L_{U^H} s,
\end{align}
where the Lie differential $L_{U^H}$ acts on ${C^\infty} (B,E)=
C^{\infty}(M, \Lambda^i ( T^* Z)\otimes F)$.
 \end{defn}

If $U_1, U_2$ are vector fields on $B$, put
\begin{align}\label{a03}
T(U_1, U_2)= -P^{TZ} [U_1^H, U_2^H] \in C^\infty (M, TZ).
\end{align}
We denote by $i_T \in \Omega^2 (B, \mbox{Hom} (E^\bullet
,E^{\bullet -1} ))$   the 2-form on $B$ which, to vector fields
$U_1, U_2$ on $ B$,
 assigns the operation of interior multiplication by $T(U_1, U_2)$ on $E$.

Let $d^Z$ be the exterior differentiation along fibers. We
consider $d^Z$ to be an element of $C^{\infty}(B, \mbox{Hom}
(E^\bullet ,E^{\bullet+1} ))$. The exterior differentiation
operator $d^M$, acting on $ \Omega (M,F)=C^\infty(M, \Lambda (T^*
M)\otimes F)$, has degree $1$ and satisfies $(d^M)^2=0$. By
\cite[Proposition 3.4]{BL}, we have
\begin{align}\label{a02}
d^M = d^Z + \nabla^E + i_T.
\end{align}
So $d^M$ is a flat superconnection of total degree $1$ on $E$. We
have
\begin{align}\label{a020}
\left(d^{Z}\right)^2=0, \quad \left[\nabla^E, d^Z\right]=0.
\end{align}

Let $g^{TZ}$ be a metric on $TZ$. Let $h^F$ be a Hermitian metric
on $F$. Let $\nabla^{F*}$ be the adjoint of $\nabla^F$ with
respect to $h^F$. Let $\omega(F,h^F)$ and $\nabla^{F,e}$ be the
$1$-form on $M$ and the connection on $F$ defined as in
(\ref{2.3}), (\ref{2.7}).

 Let $o(TZ)$ be  the orientation bundle of $TZ$,
a flat  real line bundle on $M$. Let $dv_Z$ be the Riemannian
volume  form on fibers $Z$ associated  to the metric $g^{TZ}$
(Here $dv_Z$ is viewed as a section of  $\Lambda ^{\dim Z}
(T^{*}Z)\otimes o(TZ)$). Let $\left \langle \  , \  \right
\rangle_{\Lambda (T^{*}Z)\otimes F} $ be the metric on $\Lambda
(T^{*}Z)\otimes F$ induced by $g^{TZ}, h^F$.
Then $E$ acquires a Hermitian metric $h^E$ such that for $
\alpha, \alpha' \in C^\infty (B, E)$ and $b \in B$,
\begin{align}\label{a1}
\left \langle \alpha, \alpha' \right \rangle_{h^E} (b)=
\int_{Z_b} \left \langle {\alpha, \alpha' } \right \rangle
_{\Lambda (T^*Z)\otimes F} dv_{Z_b}.
\end{align}

Let $\nabla^{E*}$, $d^{Z*}$, $(d^M)^*$, $(i_T)^*$  be the formal
adjoints of $\nabla^{E}$, $d^Z$, $d^M$, $i_T$ with respect to the
scalar product $\left \langle \, ,  \, \right \rangle_{h^E} $. Set
\begin{align}\label{a2}
&D^Z = d^Z + d^{Z*}, \qquad \quad
\nabla^{E, e}= {1 \over 2} \left(\nabla ^E + \nabla^{E*}\right),\\
&\omega(E, h^E) = \nabla^{E*}- \nabla ^E. \nonumber
\end{align}

Let $N_Z$ be the number operator of $E$, i.e. $N_Z$ acts by
multiplication by $k$ on $C^\infty (M, \Lambda^k (T^{*}Z) \otimes
F)$. For $u>0$, set
\begin{align}\label{a3}
&C_u'= u^{N_Z/2} d^M u^{-N_Z/2}, \quad C_u''= u ^{-N_Z/2}\left(d^M\right)^* u^{N_Z/2},\\
&C_u = {1 \over 2} \left(C_u'+C_u''\right), \quad D_u  = {1 \over
2} \left(C_u''-C_u'\right). \nonumber
\end{align}
Then $C_u''$ is the adjoint of $C_u'$ with respect to $h^E$.
 Moreover, $C_u$ is a superconnection on $E$ and $D_u$ is an odd  element of
 $C^\infty(B, \mbox{End} (E))$, and
\begin{align}\label{a4}
C_u^2 = - D_u^2,\quad   \left[C_u, D_u\right]=0.
\end{align}

Let  $g^{TB}$ be a Riemannian metric on $TB$. Then $g^{TM}= g^{TZ}
\oplus \pi^* g^{TB}$  is a metric on $TM$.
 Let $\nabla^{TM}$, $\nabla^{TB}$
 denote the corresponding
Levi-Civita connections  on $TM, TB$. Put $\nabla^{TZ}= P^{TZ}
\nabla^{TM}$, a connection on $TZ$. As shown in \cite[Theorem
1.9]{B}, $\nabla^{TZ}$ is independent of the choice of $g^{TB}$.
Then ${^0 \nabla} = \nabla^{TZ}\oplus  \pi ^* \nabla^{TB}$ is also
a connection on $TM$. Let $S=\nabla^{TM}- {^0 \nabla}$. By
\cite[Theorem 1.9]{B}, $\left \langle S(\cdot)\cdot,\cdot\right
\rangle_{g^{TM}}$ is a tensor independent of $g^{TB}$. Moreover,
for $U_1, U_2 \in TB$, $X,Y  \in TZ$,
\begin{align}\label{0a4}
&\left\langle S\left(U_1^H\right)X, U_2^H \right\rangle_{g^{TM}} =
-\left\langle S\left(U_1^H\right)U_2^H, X \right\rangle_{g^{TM}}\\
&\hspace*{25mm}= \left\langle S(X)U_1^H, U_2^H
\right\rangle_{g^{TM}}
=\frac{1}{2} \left\langle T\left(U_1^H, U_2^H\right), X \right\rangle_{g^{TM}}, \nonumber\\
&\left\langle S(X)Y, U_1^H \right\rangle_{g^{TM}}= - \left\langle
S(X) U_1^H, Y \right\rangle_{g^{TM}} =\frac{1}{2}
\left(L_{U_1^H}g^{TZ}\right) (X,Y), \nonumber
\end{align}
and all other terms are zero.

Let $\{f_{\alpha}\}_{\alpha=1}^p$ be an orthonormal basis of $TB$,
set $\{f^{\alpha}\}_{\alpha=1}^p$ the dual basis of $T^*B$. In the
following, it's convenient to identify $f_\alpha$ with
$f_\alpha^H$. Let $\{e_i\}_{i=1}^n$ be an orthonormal basis of
$(TZ, g^{TZ})$. We define a horizontal $1$-form $k$ on $M$ by
\begin{align}\label{a5}
k (f_{\alpha})=- \sum_i \left \langle S(e_i)e_i, f_{\alpha}\right
\rangle.
\end{align}
Set
\begin{align}\label{a6}
c(T)= {1 \over 2} \sum_{\alpha,\beta} f^\alpha\wedge f^{\beta}
c  \left(T\left(f_{\alpha}, f_{\beta}\right)   \right),\\
\widehat{c}(T)= {1 \over 2} \sum_{\alpha,\beta} f^\alpha\wedge
f^{\beta} \widehat{c} \left(T\left(f_{\alpha}, f_{\beta}\right)
\right).\nonumber
\end{align}

Let $\nabla^{\Lambda (T^* Z)}$ be the connection on $\Lambda (T^*
Z)$ induced by $\nabla ^{TZ}$.  Let $\nabla^{TZ\otimes F, e} $ be
the connection on $\Lambda (T^* Z)\otimes F$ induced by
$\nabla^{\Lambda (T^* Z)}$, $\nabla^{F,e}$. Then  by \cite[(3.36),
(3.37), (3.42)]{BL},
\begin{align}\label{a7}
&D^Z= \sum_j  c(e_j)\nabla^{TZ\otimes F, e} _{e_j} - {1 \over
2}\sum_j
\widehat{c}(e_j) \omega\left(F, h^F)(e_j\right),\\
&d^{Z*}- d^Z = -\sum_j \widehat{c}(e_j)\nabla^{TZ\otimes F, e}
_{e_j}
 + {1 \over 2}\sum_j c(e_j) \omega\left(F, h^F\right)(e_j),\nonumber\\
&\nabla^{E, e} =\sum_{\alpha} f^{\alpha} \left(\nabla^{TZ\otimes
F, e}_{f_{\alpha}} + {1 \over 2}
k\left(f_{\alpha}\right)\right ),\nonumber\\
&\omega\left(E, h^E\right) = \sum_{\alpha} f^{\alpha}
\left(\sum_{i,j}\left \langle S(e_i)e_j, f_{\alpha}\right \rangle
c(e_i) \widehat{c}(e_j) + \omega(F,
h^F)(f_{\alpha})\right).\nonumber
\end{align}
By \cite[Proposition 3.9]{BL}, one has
\begin{align}\label{a8}
&C_u= {\sqrt{u}  \over 2} D^Z + \nabla ^{E,e} - {1 \over 2\sqrt{u}} c(T),\\
&D_u = {\sqrt{u}  \over 2} \left(d^{Z*}- d^Z\right) + {1 \over
2}\omega\left(E, h^E\right) -  {1 \over 2\sqrt{u}}
\widehat{c}(T).\nonumber
\end{align}

\subsection{Deformed sub-signature operators on a fibered manifold}
\label{s3.2}

$\quad$ We assume now that  $TB$ is oriented.

Let $(\mu, h^\mu)$ be a Hermitian complex vector bundle over $B$
carrying a Hermitian connection $\nabla^\mu$.

Let $N_B, N_M$ be the number operators on $\Lambda (T^*B), \Lambda
(T^*M)$, i.e. they act as multiplication by $k$ on $\Lambda
^k(T^*B), \Lambda  ^k(T^*M)$ respectively. Then $N_M=N_B+N_Z$.

Let $\nabla^{\Lambda(T^*M)}$ be the connection on $\Lambda(T^*M)$
canonically induced from $\nabla^{TM}$. Let $\nabla
^{\Lambda(T^*M)\otimes \pi^*\mu\otimes F}$ (resp. $\nabla
^{\Lambda(T^*M)\otimes \pi^*\mu\otimes F,e}$)  be the tensor
product connection on $\Lambda(T^*M)\otimes  \pi^*\mu\otimes F$
induced by $\nabla^{\Lambda(T^*M)}$, $\pi^*\nabla^\mu$ and
$\nabla^{F}$ (resp. $\nabla^{F,e}$).

Let $\{e_a\}_{a=1}^{m}$ be an
 orthonormal basis of $TM$, and its dual basis $\{e ^a\}_{a=1}^{m}$.
 Let
$\{f_\alpha\}_{\alpha=1}^p$  be an oriented orthonormal basis of
$TB$. Set
\begin{align} \label{a17}
&\tau (TB)= (\sqrt{-1})^{p(p+1)\over
2} c\left(f_1^H\right)\cdots c\left(f_p^H\right),  \\
&\tau = (-1)^{N_Z}\tau (TB). \nonumber
\end{align}
Then the operators $  \tau (TB), \tau$ act naturally on $\Lambda
(T^*M)$, and
\begin{align}\label{a18}
\tau (TB)^2=\tau ^2 =1.
\end{align}

Let $d^{\nabla^\mu} : \Omega ^a (M, \pi^*\mu\otimes F)\to \Omega
^{a+1} (M, \pi^*\mu\otimes F)$  be the unique extension of $\nabla
^\mu, \nabla ^F$ which satisfies the Leibniz rule. Let
$d^{\nabla^\mu *} $ be the adjoint of $d^{\nabla^\mu}$ with
respect to the scalar product $\left \langle \ ,\ \right
\rangle_{\Omega (M, \pi^*\mu\otimes F)}$ on $\Omega (M,
\pi^*\mu\otimes F)$ induced by $g^{TM}, h^\mu, h^F$ as in
(\ref{a1}).
 As in \cite[(4.26), (4.27)]{BZ}, we have
\begin{align}\label{a19}
&d^{\nabla^\mu} = \sum_a e^a \wedge \nabla ^{\Lambda(T^*M)\otimes \pi^*\mu\otimes F}_{e_a},\\
&d^{\nabla^\mu *} =- \sum_a i_{e_a} \wedge \left(\nabla
^{\Lambda(T^*M)\otimes \pi^*\mu\otimes F}_{e_a} + \omega \left(F,
h^F\right)(e_a)\right).\nonumber
\end{align}

Following \cite{Z}, let $\widetilde{\nabla}^{\Lambda(T^*M)}$ be
the Hermitian connection on $\Lambda(T^*M)$ defined by (cf.
\cite[(1.21)]{Z})
\begin{align}\label{a22}
\widetilde{\nabla}^{\Lambda(T^*M)}_X=\nabla^{\Lambda(T^*M)}_X-{1\over
2}\sum_{\alpha=1}^{p}\widehat{c}\left(P^{TZ}S(X)f_\alpha\right)\widehat{c}\left(f_\alpha\right),\
\ \ X\in TM.
\end{align}

 Let $\widetilde{\nabla}^{e}$ be the tensor product connection on
$\Lambda(T^*M)\otimes  \pi^*\mu\otimes F$ induced by
$\widetilde{\nabla}^{\Lambda(T^*M)}$, $\pi^*\nabla^\mu$ and
$\nabla^{F,e}$.  Following \cite[(3.23)]{MZ}, for any $r\in
{\mathbb R}$, set
\begin{align}\label{a23}
 &D^{\pi^*\mu\otimes F} = \sum_{a=1}^{m}c(e_a)
\widetilde{\nabla}^{e}_{e_a}
-\frac{1}{2} \sum_{i=1}^{n} \widehat{c}(e_i) \omega\left(F, h^F\right)(e_i),\\
 &\widehat{D} ^{\pi^*\mu\otimes F} =- \sum_{i=1}^{n}\widehat{c}(e_i)
\widetilde{\nabla}^{e}_{e_i}
+ \frac{1}{2} \sum_{a=1}^{m} c(e_a) \omega\left(F, h^F\right)(e_a)\nonumber\\
&\hspace*{30mm}- \frac{1}{4}\sum_{\alpha,\beta=1}^p
\widehat{c}\left(T\left(f_\alpha, f_\beta\right)\right)
\widehat{c}(f_\alpha)\widehat{c}\left(f_ \beta\right) ,\nonumber\\
&D^{\pi^*\mu\otimes F}(r) = D^{\pi^*\mu\otimes F} + \sqrt{-1} r
\widehat{D} ^{\pi^*\mu\otimes F}. \nonumber
\end{align}

From (\ref{a23}), the operators $D^{\pi^*\mu\otimes F}$,
$D^{\pi^*\mu\otimes F}(r)$ are formally self-adjoint first order
elliptic operators,  and $\widehat{D}^{\pi^*\mu\otimes F}$ is a
skew-adjoint first order differential operator. Moreover, the
operator $D^{\pi^*\mu\otimes F}$ is locally of Dirac type.

 By \cite[(3.20) and Proposition 3.4]{MZ}, one has
 \begin{align}\label{a20}
& D^{\pi^*\mu\otimes F}  =\frac{1}{2} \left[
\left(d^{\nabla^\mu}+d^{\nabla^\mu *}\right)
+ (-1)^{p+1} \tau\left(d^{\nabla^\mu}+d^{\nabla^\mu *}\right) \tau \right],\\
&\widehat{D}^{\pi^*\mu\otimes F}  =\frac{1}{2}\left[
\left(d^{\nabla^\mu *}-d^{\nabla^\mu}\right) + (-1)^{p+1} \tau
\left(d^{\nabla^\mu *}-d^{\nabla^\mu}\right) \tau
\right],\nonumber
\end{align}
which partly explains the motivation of introducing these
operators (compare with (\ref{2.22})).

 By (\ref{a17}), (\ref{a18}) and (\ref{a20}), one
verifies (cf. \cite[(3.28)]{MZ})
\begin{align}\label{a28}
&\tau D^{\pi^*\mu\otimes F} = (-1)^{p+1} D^{\pi^*\mu\otimes
F}\tau,\quad &\tau \widehat{D}^{\pi^*\mu\otimes F} = (-1)^{p+1}
\widehat{D}^{\pi^*\mu\otimes F}\tau.
\end{align}

\begin{Rem} \label{w1}It is important to note that by (\ref{a28}), when $p=\dim B$ is
even, both $ D^{\pi^*\mu\otimes F}$ and
$\widehat{D}^{\pi^*\mu\otimes F}$ anti-commute with $\tau$.
\end{Rem}

\begin{Rem} When $\mu=F={\bf C}$ and $p=\dim B$ is even, $ D^{\pi^*\mu\otimes
F}$ has been constructed in \cite{Z} and \cite{Z1}, where it is
called the sub-signature operator.
\end{Rem}

\subsection{Bismut-Freed connection of the deformed family}\label{s3.3}

$\quad$ We assume that $p=\dim B$ is even. Moreover, we make the following
technical assumption.

$\ $

\noindent {\bf Technical assumption}. The flat vector bundle $F$
over $M$ is fiberwise acyclic, that is $H^*(Z_b,F|_{Z_b})=\{0\}$
on each fiber $Z_b$, $b\in B$.

$\ $

For any $\varepsilon>0$, we change $g^{TB}$ to ${1\over
\varepsilon}g^{TB}$ and do everything again for
$g_\varepsilon^{TM}=g^{TZ}\oplus {1\over \varepsilon}\pi^*g^{TB}$.
We will use a subscript $\varepsilon$ to denote the resulting
objects.

For any $r\in{\mathbb R}$, one verifies directly that the
coefficients of ${1\over \sqrt{\varepsilon}}$ in ${1\over
\sqrt{\varepsilon}}D^{\pi^*\mu\otimes F}_\varepsilon(r)$ is given by
$ d^Z+d^{Z*}-\sqrt{-1}r\left(d^Z-d^{Z*}\right). $ Since
\begin{align}\label{p2}
\left(d^Z+d^{Z*}-\sqrt{-1}r\left(d^Z-d^{Z*}\right)\right)^2=\left(1+r^2\right)\left(d^Z+d^{Z*}\right)^2,
\end{align}
by proceeding as in \cite{BC}, one sees that when $\varepsilon>0$ is
small enough, $D^{\pi^*\mu\otimes F}_\varepsilon(r)$ is invertible
near $r$. In fact, the eigenvalues of $D^{\pi^*\mu\otimes
F}_\varepsilon(r)$ are uniformly bounded away from zero.

Consider now $D^{\pi^*\mu\otimes F}_\varepsilon(r)$ as an
$r$-family which anti-commutes with the ${\mathbb Z}_2$-grading
defined by $\tau$.

Then one can construct the Quillen determinant line bundle over $r$
and the associated Bismut-Freed connection on it (cf. \cite{BF}).
Moreover,  by the above discussion and by \cite[3.8]{BF}, we know
that when $\varepsilon>0$ is small enough, the imaginary part of the
Bismut-Freed connection form is given by
\begin{eqnarray*}
\ \ \ \ \ \ \ \ \ \ &  & {1\over 2\sqrt{-1}}\ {\rm F.P.}
\int_0^{+\infty}{\rm Tr}_s\left[D^{\pi^*\mu\otimes
F}_\varepsilon(r){\partial D^{\pi^*\mu\otimes
F}_\varepsilon(r)\over\partial r}e^{-t \left(D^{\pi^*\mu\otimes
F}_\varepsilon(r)\right)^2}\right] dt \nonumber \\
 & = & {1\over 2 }\ {\rm F.P.} \int_0^{+\infty}{\rm Tr}_s\left[D^{\pi^*\mu\otimes
F}_\varepsilon(r){ \widehat{ D}^{\pi^*\mu\otimes F}_\varepsilon
 }e^{-t \left(D^{\pi^*\mu\otimes
F}_\varepsilon(r)\right)^2}\right] dt \nonumber \\
 & = & -{1\over 2 }\ {\rm F.P.} \int_0^{+\infty}{\rm Tr}_s\left[{ \widehat{
D}^{\pi^*\mu\otimes F}_\varepsilon
 }D^{\pi^*\mu\otimes
F}_\varepsilon(r)e^{-t \left(D^{\pi^*\mu\otimes
F}_\varepsilon(r)\right)^2}\right]dt  ,
\end{eqnarray*}
where the supertrace is with respect to $\tau$ and 'F.P.' means
taking the finite part of the (divergent) integral. As usual we use
the zeta function regularization. Thus, we define
\begin{align} \label{ps3}
\delta_\varepsilon(F, r)( s)= -\frac{1}{2\Gamma(s)}
\int_0^{+\infty} t^s {\rm Tr}_s\left[{ \widehat{ D}^{\pi^*\mu\otimes
F}_\varepsilon
 }D^{\pi^*\mu\otimes
F}_\varepsilon(r)e^{-t \left(D^{\pi^*\mu\otimes
F}_\varepsilon(r)\right)^2}\right]dt.
\end{align}
\newline

{\em Remark}. Note that we have built the factor $\frac{1}{2}$ into
the definition (unlike the finite dimensional case).
\newline

In the next subsection we will study the asymptotic expansions of
the integrand in (\ref{ps3}) which implies that the integral,
convergent for $\Re\, s$ sufficiently large, has meromorphic
continuation to the whole complex plane with $s=0$ a regular point.
Therefore we define our invariant by
\begin{align} \label{p3}
\delta_\varepsilon(F)(r)=\delta_\varepsilon(F, r)'( 0).
\end{align}

 Also in the next section we
will compute the adiabatic limit of $\delta_\varepsilon(F)(r)$ as
$\varepsilon\rightarrow 0$.

We remark that the definition of $\delta_\varepsilon(F)( r)$ does
not make use of the technical assumption $H^*(Z_b,F|_{Z_b})=\{0\}$.

\subsection{The adiabatic limit and the torsion form}\label{s3.4}

$\quad$ We begin with a lemma.

\begin{lemma}  $\widehat{D}^{\pi^*\mu\otimes F}$ is the quantization of $D_4$; namely, it is
obtained by replacing the horizontal differential forms in $D_4$ by
the corresponding Clifford multiplications. Similarly,
$D^{\pi^*\mu\otimes F}$ is the quantization of $C_4$.
\end{lemma}
\begin{proof} By (\ref{a23}),
\[ \widehat{D} ^{\pi^*\mu\otimes F} =- \sum_{i=1}^{n}\widehat{c}(e_i)
\widetilde{\nabla}^{e}_{e_i} + \frac{1}{2} \sum_{a=1}^{m} c(e_a)
\omega\left(F, h^F\right)(e_a)- \frac{1}{4}\sum_{\alpha,\beta=1}^p
\widehat{c}\left(T\left(f_\alpha, f_\beta\right)\right)
\widehat{c}(f_\alpha)\widehat{c}\left(f_ \beta\right),\] where the
connection $\widetilde{\nabla}^{\Lambda(T^*M)}$ is defined by
(\ref{a22}). Thus, we now look at the connection in a bit more
detail. Since $\nabla^{TM}= \ ^0\nabla +S$ and $^0\nabla =
\nabla^{TZ} \oplus \pi^*\nabla^{TB}$, we find (for simplicity, we
denote $S_{ij\alpha}=\langle S(e_i)e_j, \ f_\alpha \rangle$ and so
on)
\[  \nabla^{\Lambda(T^*M)}_{e_i} =
\nabla^{\Lambda(T^*Z)}_{e_i} -\frac{1}{4} \left[ S_{ij\alpha} \left(
\widehat{c}(e_j)\widehat{c}(f_\alpha) - c(e_j) c(f_\alpha) + c(e_j)
\widehat{c}(f_\alpha) -  \widehat{c}(e_j)c(f_\alpha) \right) \right.
\] \[ \hspace{.05in}   +  S_{i\alpha j}
(\widehat{c}(f_\alpha) \widehat{c}(e_j) - c(f_\alpha) c(e_j)  +
c(f_\alpha) \widehat{c}(e_j)- \widehat{c}(f_\alpha) c(e_j) ) \] \[
\hspace{.2in} \left. +  S_{i\alpha\beta} (\widehat{c}(f_\alpha)
\widehat{c}(f_\beta) - c(f_\alpha) c(f_\beta) +  c(f_\alpha)
\widehat{c}(f_\beta)- \widehat{c}(f_\alpha) c(f_\beta)) \right]
\]
\[ = \nabla^{\Lambda(T^*Z)}_{e_i} -\frac{1}{2} S_{i\alpha j}
\left[  \widehat{c}(f_\alpha) \widehat{c}(e_j) - c(f_\alpha) c(e_j)
\right] \hspace{1.1in}
\]
\[ \hspace{.3in} -\frac{1}{4}S_{i\alpha\beta} \left[ \widehat{c}(f_\alpha)
\widehat{c}(f_\beta) - c(f_\alpha) c(f_\beta) +  c(f_\alpha)
\widehat{c}(f_\beta)- \widehat{c}(f_\alpha) c(f_\beta)\right]. \]
Hence
\[ \widetilde{\nabla}^{\Lambda(T^*M)}_{e_i}=
\nabla^{\Lambda(T^*Z)}_{e_i} -\frac{1}{2}  S_{ij\alpha} c(f_\alpha)
c(e_j) \hspace{2.7in}
\]
\[ \hspace{.3in} -\frac{1}{4}S_{i\alpha\beta} \left[ \widehat{c}(f_\alpha)
\widehat{c}(f_\beta) - c(f_\alpha) c(f_\beta) +  c(f_\alpha)
\widehat{c}(f_\beta)- \widehat{c}(f_\alpha) c(f_\beta)\right]. \]
Therefore,
\[ \widehat{c}(e_i) \widetilde{\nabla}^{\Lambda(T^*M)}_{e_i}=
\widehat{c}(e_i) \nabla^{\Lambda(T^*Z)}_{e_i}  \hspace{3.5in}
\] \[ \hspace{1in} -\frac{1}{4} S_{i\alpha\beta} \widehat{c}(e_i) \left[ \widehat{c}(f_\alpha)
\widehat{c}(f_\beta) - c(f_\alpha) c(f_\beta) +  c(f_\alpha)
\widehat{c}(f_\beta)- \widehat{c}(f_\alpha) c(f_\beta) \right].
\]
And so
\[ \widehat{c}(e_i) \widetilde{\nabla}^{e}_{e_i} = \widehat{c}(e_i) \nabla^{TZ\otimes F, e}_{e_i} -\frac{1}{8}
\widehat{c}(T(f_\alpha, f_\beta))\widehat{c}(f_\alpha)
 \widehat{c}(f_\beta) + \frac{1}{8}
\widehat{c}(T(f_\alpha, f_\beta)) c(f_\alpha) c(f_\beta) \]
\[
-\frac{1}{8} \widehat{c}(T(f_\alpha, f_\beta)) [ c(f_\alpha)
\widehat{c}(f_\beta)- \widehat{c}(f_\alpha) c(f_\beta) ].
\hspace{1in}
\]
The last term here, $-\frac{1}{8} \widehat{c}(T(f_\alpha, f_\beta))
[ c(f_\alpha) \widehat{c}(f_\beta)- \widehat{c}(f_\alpha)
c(f_\beta)]$, vanishes by the antisymmetry. Using the formula above
together with (\ref{a7}), (\ref{a8}), (\ref{a23}), we prove our
lemma. (The other case is dealt with similarly.)
\end{proof}
\begin{prop}\label{t3.4} We have
\begin{multline}\label{3.44}
 \lim_{\epsilon\rightarrow 0} {\rm Tr}_s\left[{ \widehat{
D}^{\pi^*\mu\otimes F}_\varepsilon }D^{\pi^*\mu\otimes
F}_\varepsilon(r)e^{-t \left(D^{\pi^*\mu\otimes
F}_\varepsilon(r)\right)^2}\right]
  =  - 
  \int_B L\left(TB, \nabla^{TB}\right)\ch\left(\mu,
  \nabla^{\mu}\right)\\ \cdot
 \varphi\,{\rm Tr}_s\left[t^{-\frac{1}{2}}D_{4t} \frac{\partial}{\partial \sqrt{t}}
 \left( C_{4t} + \sqrt{-1} r D_{4t} \right)
   e^{- (1+r^2)C_{4t}^2} \right]  .\end{multline}
\end{prop}
\begin{proof}
 Again, we introduce an auxiliary Grassmann variable $z$ and rewrite
\begin{align}\label{3.45} {\rm Tr}_s \left[ \widehat{
D}^{\pi^*\mu\otimes F}_\varepsilon D^{\pi^*\mu\otimes
F}_\varepsilon(r)e^{-t \left(D^{\pi^*\mu\otimes
F}_\varepsilon(r)\right)^2}\right] = \hspace{2in} \\
\hspace{1in}  -{\rm Tr}_{s, z}\left[ t^{-\frac{1}{2}} \widehat{
D}^{\pi^*\mu\otimes F}_\varepsilon e^{-t \left(D^{\pi^*\mu\otimes
F}_\varepsilon(r)\right)^2+ z\sqrt{t} D^{\pi^*\mu\otimes
F}_\varepsilon(r)} \right] .\nonumber
\end{align}

The following Lichnerowicz formula was proved in \cite[Theorem
1.1]{Z}.
\begin{multline}\label{b7}
\left( D^{\pi^*\mu\otimes F}\right)^2 =-\wi{\Delta}^{e} +\frac{K}{4}
+ \frac{1}{2}\sum_{a,b=1}^{m}c(e_a)c(e_b)
 (\widehat{R}^e + \pi^* R^\mu) (e_a,e_b) \\
 +{1\over 4}\sum_{i=1}^{n}
\left(\omega\left(F,h^F\right)(e_i)\right)^2 +{1\over
8}\sum_{i,j=1}^{n}\widehat{c}(e_i)\widehat{c}(e_j)
\left(\omega\left(F,h^F\right)\right)^2 (e_i,e_j)\\
-\frac{1}{2} \sum_{a=1}^{m}c(e_a) \Big [\sum_{i=1}^{n}
\widehat{c}(e_i)\nabla ^{TM\otimes F,e}_{e_a}\omega \left(F,h^F\right)(e_i)\\
    + \sum_{\alpha=1}^{p} \widehat{c}(f_\alpha)
 \omega \left(F,h^F\right)(P^{TZ} S(e_a) f_\alpha)\Big ] .
 \end{multline}

Similarly, the following formulas are shown in \cite[Proposition
3.6]{MZ}.
\begin{multline}\label{b8}
 \left(\widehat{D}^{\pi^*\mu\otimes F}\right)^2 =
\sum_{i=1}^n \left( (\wi{\nabla}^{e}_{e_i} )^2 -\wi{\nabla}
^{e}_{\nabla ^{TM}_{e_i}e_i}\right)   + \frac{1}{2}
\sum_{i,j=1}^{n}\widehat{c}(e_i)\widehat{c}(e_j)(\wi{\nabla}^{e})^2
(e_i,e_j)\\
+ \frac{1}{4}\sum_{i=1}^n\widehat{c}(e_i) \Big
[\wi{\nabla}^{e}_{e_i},  \sum_{\alpha,\beta=1}^p
\widehat{c}(T(f_\alpha, f_\beta))
\widehat{c}(f_\alpha) \widehat{c}(f_ \beta)\Big]\\
+\frac{1}{2}\sum_{\alpha,\beta=1}^p  \widehat{c}(f_\alpha)
\widehat{c}(f_ \beta) \wi{\nabla}^{e}_{T(f_\alpha, f_\beta)} -
\frac{1}{2}\sum_{i=1}^n\sum_{a=1}^m\widehat{c}(e_i)c(e_a)(\nabla
^{TM\otimes F,e}_{e_i}
\omega \left(F,h^F\right))(e_a)\\
-{1\over
4}\sum_{a=1}^{m}\left(\omega\left(F,h^F\right)(e_a)\right)^2
+{1\over 8}\sum_{a,b=1}^{m}c(e_a)c(e_b)
\left(\omega\left(F,h^F\right)\right)^2 (e_a,e_b)\\
+\frac{1}{16} \Big(\sum_{\alpha,\beta=1}^p \widehat{c}(T(f_\alpha,
f_\beta))\widehat{c}(f_\alpha) \widehat{c}(f_ \beta)\Big)^2,
 \end{multline}
\begin{multline}\label{b9}
[D^{\pi^*\mu\otimes F}, \widehat{D}^{\pi^*\mu\otimes F}]=
-\sum_{a=1}^m\sum_{i=1}^n c(e_a)\widehat{c}(e_i)
\left(\widehat{R}^e+\pi^* R^\mu
+\frac{1}{4}\omega \left(F,h^F\right)^2\right) (e_a,e_i)\\
-\sum_{\alpha=1}^{p}\omega
\left(F,h^F\right)(f_\alpha)\wi{\nabla}^{e}_{f_\alpha}
 + \frac{1}{4}\sum_{\alpha,\beta=1}^p\omega \left(F,h^F\right)( T(f_\alpha, f_\beta))
 \widehat{c}(f_\alpha)\widehat{c}(f_\beta)
\\
+ \frac{1}{4} \sum_{a=1}^mc(e_a) \Big [\wi{\nabla}^{e}_{e_a},
\sum_{\alpha,\beta=1}^p \widehat{c}(T(f_\alpha,
f_\beta))\widehat{c}(f_\alpha) \widehat{c}(f_ \beta)\Big].
 \end{multline}

These formulas show that there are no second order fiberwise
differentiation in $\left(\widehat{D}^{\pi^*\mu\otimes F}\right)^2$
and $[D^{\pi^*\mu\otimes F}, \widehat{D}^{\pi^*\mu\otimes F}]$.
Therefore one can apply the standard Getzler rescaling to
$(\widehat{D}^{\pi^*\mu\otimes F}_{s,\var})^2$ and
$[D^{\pi^*\mu\otimes F}_{s,\var}, \widehat{D}^{\pi^*\mu\otimes
F}_{s,\var}]$ with no problem and all terms converge as $\epsilon
\rightarrow 0$.

On the other hand, in \cite[Proposition 2.2]{Z}, Zhang formulated a
Lichnerowicz type formula for $t (D^{\pi^*\mu\otimes
F}_{s,\var})^2-z\sqrt{t}D^{\pi^*\mu\otimes F}_{s,\var}$. The only
singular term (for Getzler's rescaling) as $\var \to 0$ appears in
\begin{align}\label{c15}
-t\var \sum_{\alpha} \Big(\wi{\nabla}_{f_\alpha}
+\frac{\sqrt{\var}}{2} \sum_{i,\beta}\langle S(f_\alpha)e_i,f_\beta
\rangle c(e_i)c(f_\beta) +\frac{zc(f_\alpha)}{2\sqrt{t\var}} \Big)
^2.
\end{align}
This singular term can be easily eliminated by the exponential
transform, namely conjugating by the exponential \begin{equation}
\label{exp} e^{\frac{z\sum_{\alpha=1}^p y_{\alpha}
c(f_{\alpha})}{2\sqrt{t\epsilon}}}. \end{equation} We then do the
Getzler rescaling $G_{\sqrt{t\var}}$:
\[ y_\alpha \rightarrow
\sqrt{t\var}y_\alpha, \ \ \ \ \ \nabla_{f_\alpha}\rightarrow
\frac{1}{\sqrt{t\var}}\nabla_{f_\alpha}, \ \ \ \ \
c(f_\alpha)\rightarrow \frac{1}{\sqrt{t\var}}f^\alpha \wedge
-\sqrt{t\var}i_{f_\alpha}. \] By  (\ref{a4}), (\ref{0a4}),
(\ref{a23}), (\ref{c15}), \cite[Proposition 2.2]{Z}, and by
proceeding similarly as in \cite[(4.69)]{BC}, after the conjugation
by (\ref{exp}), the $G_{\sqrt{t\var}}$ rescaled operator of
$t(D^{\pi^*\mu\otimes F}_{s,\var}(r))^2 -z \sqrt{t}
D^{\pi^*\mu\otimes F}_{s,\var}(r)$ converges as
 $\var \to 0$ to
\begin{align}
 {\cal H} + (1+r^2) (C_{4t}^\mu)^2
-z \left(\sqrt{t}D^Z+  \frac{c(T)}{4\sqrt{t}}   +   \sqrt{-1}r
\left(\sqrt{t}(d^{Z*} -d^Z) + \frac{\widehat{c}(T)}{4\sqrt{t}}
\right)\right) \nonumber \\
={\cal H} + (1+r^2) (C_{4t}^2+R^\mu) -z 2t \frac{\partial}{\partial
t}(C_{4t} + \sqrt{-1}r D_{4t}), \label{c19}
\end{align} where
\[ {\cal H}=-\sum_{\alpha} \left( \nabla_{f_\alpha}
+\frac{1}{4}\left\langle R^{TB}_{b_0}y, f_\alpha\right\rangle
\right) ^2 - \frac{1}{4} \sum_{\alpha,\beta}\left\langle
R^{TB}_{b_0} f_{\alpha}, f_{\beta}\right\rangle
\widehat{c}(f_{\alpha})\widehat{c}(f_{\beta}) . \]

Finally, by the previous lemma, we see that the rescaled operator
obtained from the conjugation by (\ref{exp}) of $\widehat{
D}^{\pi^*\mu\otimes F}_\varepsilon$ converges to $D_{4t}$ as $\var
\to 0$. Proceeding as in \cite{MZ} and noting $2t
\frac{\partial}{\partial t}= \frac{\partial}{\partial \sqrt{t}}$, we
obtain the desired formula.
\end{proof}

\begin{prop}\label{t3.5} We have the following uniform asymptotic
expansion
\begin{multline}\label{3.45}
 {\rm Tr}_s\left[{ \widehat{
D}^{\pi^*\mu\otimes F}_\varepsilon }D^{\pi^*\mu\otimes
F}_\varepsilon(r)e^{-t \left(D^{\pi^*\mu\otimes
F}_\varepsilon(r)\right)^2}\right]
  =  c_{-k}(\var)\, t^{-k} + c_{-k + 1}(\var)\, t^{-k+1} + \cdots  , \end{multline}
  where $k=\frac{3}{2}$ if $n$ (dimension of the fiber) is odd and $k=1$ if $n$ is even.
Similarly,
\begin{multline}\label{3.46}
 - \int_B L\left(TB, \nabla^{TB}\right)\ch\left(\mu,
  \nabla^{\mu}\right) \varphi\,{\rm Tr}_s\left[t^{-\frac{1}{2}}D_{4t} \frac{\partial}{\partial \sqrt{t}}
 \left( C_{4t} + \sqrt{-1} r D_{4t} \right)   e^{- (1+r^2)C_{4t}^2} \right] \\
=  c_{-k}\, t^{-k} +  c_{-k+1}\, t^{-k+1} + \cdots .\end{multline}
Moreover,
\[ c_{i/2}(\var) \longrightarrow
c_{i/2} \ \ \  {\rm as} \ \var \rightarrow 0. \]
\end{prop}
\begin{proof}
Using two auxiliary Grassmann variables $z_1, z_2$, we write
\begin{align}\label{3.45} {\rm Tr}_s \left[ \widehat{
D}^{\pi^*\mu\otimes F}_\varepsilon D^{\pi^*\mu\otimes
F}_\varepsilon(r)e^{-t \left(D^{\pi^*\mu\otimes
F}_\varepsilon(r)\right)^2}\right] = \hspace{2in} \\
- t^{-2}{\rm Tr}_{s, z_1, z_2}\left[ e^{-t \left([D^{\pi^*\mu\otimes
F}_\varepsilon(r)]^2 -z_1 D^{\pi^*\mu\otimes F}_\varepsilon(r)- z_2
\widehat{ D}^{\pi^*\mu\otimes F}_\varepsilon \right)} \right]
\nonumber .
\end{align}
Applying the standard elliptic theory to the right hand side of
(\ref{3.45}), we derive an asymptotic expansion \begin{multline}
\hspace{.1in} {\rm Tr}_s \left[ \widehat{ D}^{\pi^*\mu\otimes
F}_\varepsilon D^{\pi^*\mu\otimes F}_\varepsilon(r)e^{-t
\left(D^{\pi^*\mu\otimes F}_\varepsilon(r)\right)^2}\right] =
c_{-m/2 -2}(\var)\, t^{-m/2-2}
 \\ \hspace{0in} + c_{-m/2 -1}(\var)\, t^{-m/2-1} + \cdots.
\nonumber
\end{multline}

To prove the vanishing of the coefficients, we revert to one
auxiliary Grassmann variable $z$ and rewrite
\begin{align}\label{extra} t^{\frac{3}{2}}\ {\rm Tr}_s \left[ \widehat{
D}^{\pi^*\mu\otimes F}_\varepsilon D^{\pi^*\mu\otimes
F}_\varepsilon(r)e^{-t \left(D^{\pi^*\mu\otimes
F}_\varepsilon(r)\right)^2}\right] = \hspace{2in} \\
\hspace{1in}  -{\rm Tr}_{s, z}\left[ t\, \widehat{
D}^{\pi^*\mu\otimes F}_\varepsilon e^{-t \left(D^{\pi^*\mu\otimes
F}_\varepsilon(r)\right)^2+ z\sqrt{t} D^{\pi^*\mu\otimes
F}_\varepsilon(r)} \right] .\nonumber
\end{align}

 As usual, one fixes a point
of $M$ and employs the normal coordinates $x$ around the point.
Consider the Getzler rescaling $G_{\sqrt{t}}^M$:
\[  x_a \rightarrow
\sqrt{t}\, x_a, \ \ \ \ \ \nabla_{e_a}\rightarrow
\frac{1}{\sqrt{t}}\, \nabla_{e_a}, \ \ \ \ \ c(e_a)\rightarrow
\frac{1}{\sqrt{t}}\, e_a \wedge -\sqrt{t}\, i_{e_a}.
\]

By (\ref{b7}), (\ref{b8}), (\ref{b9}) and the same argument as in
\cite{BF}, we can formulate a Lichnerowicz formula for $t
\left(D^{\pi^*\mu\otimes F}_\varepsilon(r)\right)^2- z\sqrt{t}
D^{\pi^*\mu\otimes F}_\varepsilon(r)$. The only singular term with
respect to the Getzler rescaling $G_{\sqrt{t}}^M$ as $t \to 0$
appears in
\[ -t\var \sum_{a=1}^p \Big(\nabla_{e_a} +\frac{zc(e_a)}{2\sqrt{t\var}} \Big)^2
-t \sum_{a=p+1}^m \Big(\nabla_{e_a} +\frac{zc(e_a)}{2\sqrt{t}}
\Big)^2.
\]
This singular term can be easily eliminated by the exponential
transform, namely conjugating by the exponential
\[
\label{exp} e^{\frac{z\sum_{a=1}^m x_{a} c(e_a)}{2\sqrt{t}}}.
\]
Thus, after the exponential transform and then the Getzler rescaling
$G_{\sqrt{t}}^M$, we find that $t \left(D^{\pi^*\mu\otimes
F}_\varepsilon(r)\right)^2- z\sqrt{t} D^{\pi^*\mu\otimes
F}_\varepsilon(r)$ converges as $t\to 0$ to
\[  \mathcal H(r, {\var}) -r^2 \omega^2 + z \sqrt{-1} r \sum_{a=p+1}^m \hat{c}(e_a) \partial_a,        \]
where
\[  \mathcal H(r, {\var})=- \var \sum_{a=1}^p \left( \partial_a
+\frac{1}{4}\left\langle R^{TM}_{p_0}y, e_a\right\rangle \right) ^2
-(1+r^2) \sum_{a=p+1}^m \left( \partial_a +\frac{1}{4}\left\langle
R^{TM}_{p_0}y, e_a\right\rangle \right) ^2 . \]

On the other hand, after the exponential transform and then the
Getzler rescaling $G_{\sqrt{t}}^M$, $t\, \widehat{
D}^{\pi^*\mu\otimes F}_\varepsilon$ converges to
\[ -\frac{z}{2} \sum_{a=p+1}^m \hat{c}(e_a)\ e_a \wedge. \]
It follows that
\[  \lim_{t\rightarrow 0} t^{\frac{3}{2}} {\rm Tr}_s \left[ \widehat{ D}^{\pi^*\mu\otimes F}_\varepsilon
D^{\pi^*\mu\otimes F}_\varepsilon(r)e^{-t \left(D^{\pi^*\mu\otimes
F}_\varepsilon(r)\right)^2}\right]= -\int^B \sum_{a=p+1}^m
\hat{c}(e_a)\ e_a \wedge \ e^{-\mathcal H(\var, r) +r^2 \omega^2},
\] where $\int^B$ denotes the Berezin integral (cf. \cite[Page 604]{MZ}).

Thus, we deduce that
 $c_{i}(\var)=0$ for $-n/2 -2\leq i < -k$, with $k=\frac{3}{2}$ if $n$ is odd.
On the other hand, if $n$ is even, the Berezin integral on the right
hand side vanishes for parity reason, and thus $k=1$.

Now we show that the asymptotic expansion is uniform in $\var$.
According to the discussion above, after the conjugation by
(\ref{exp}),  the $G_{\sqrt{\var}}$ rescaled operator of
$(D^{\pi^*\mu\otimes F}_{s,\var}(r))^2 -z D^{\pi^*\mu\otimes
F}_{s,\var}(r)$ converges as
 $\var \to 0$ to
\[
 {\cal H} + (1+r^2) (C_{4t}^2+R^\mu)
-z \left(D^Z+  \frac{c(T)}{4}   +   \sqrt{-1}r \left((d^{Z*} -d^Z) +
\frac{\widehat{c}(T)}{4} \right)\right) \nonumber  .
\]
Similarly, the $G_{\sqrt{\var}}$ rescaled operator of $\widehat{
D}^{\pi^*\mu\otimes F}_\varepsilon$ converges to $D_4$.  Since the
asymptotic expansion of
\[ {\rm Tr}_s \left[ \widehat{ D}^{\pi^*\mu\otimes F}_\varepsilon
D^{\pi^*\mu\otimes F}_\varepsilon(r)e^{-t \left(D^{\pi^*\mu\otimes
F}_\varepsilon(r)\right)^2}\right] \] depends only on the local
symbols of the rescaled operators of $(D^{\pi^*\mu\otimes
F}_{s,\var}(r))^2 -z D^{\pi^*\mu\otimes F}_{s,\var}(r)$ and
$\widehat{ D}^{\pi^*\mu\otimes F}_\varepsilon$, the coefficients
$c_i(\var)$ of its asymptotic expansion converges uniformly to that
of
\[ {\rm Tr}_s \left[  D_4 e^{-t \left( {\cal H} + (1+r^2) (C_{4t}^2+R^\mu)
-z (D^Z+  \frac{c(T)}{4}   +   \sqrt{-1}r (d^{Z*} -d^Z +
\frac{\widehat{c}(T)}{4} ) ) \right)} \right] . \] On the other
hand, since
\begin{multline}
{\rm Tr}_s \left[ D_{4t} e^{-\left({\cal H} + (1+r^2) (C_{4t}^\mu)^2
-z \left(\sqrt{t}D^Z+ \frac{c(T)}{4\sqrt{t}}   +   \sqrt{-1}r
\left(\sqrt{t}(d^{Z*} -d^Z) + \frac{\widehat{c}(T)}{4\sqrt{t}}
\right)\right)\right) } \right] =  \nonumber \\
t^{-1/2} {\rm Tr}_s \left[ D_4e^{-t \left( {\cal H} + (1+r^2)
(C_{4t}^2+R^\mu) -z (D^Z+ \frac{c(T)}{4}   +   \sqrt{-1}r (d^{Z*}
-d^Z + \frac{\widehat{c}(T)}{4} ) ) \right)} \right] , \nonumber
\end{multline}
we obtain (\ref{3.46}) and also the convergence of asymptotic
coefficients.
\end{proof}

\begin{cor} The function $\delta_\varepsilon(F, r)( s)$ in
(\ref{ps3}) has a meromorphic continuation to the whole complex
plane with $s=0$ a regular point.
\end{cor}

\begin{proof} The integral in
(\ref{ps3}) is convergent at $t=\infty$ since
\[ {\rm Tr}_s\left[{ \widehat{
D}^{\pi^*\mu\otimes F}_\varepsilon }D^{\pi^*\mu\otimes
F}_\varepsilon(r)e^{-t \left(D^{\pi^*\mu\otimes
F}_\varepsilon(r)\right)^2}\right]
\]
is exponentially decaying in $t$ as $t \rightarrow \infty$. On the
other hand, it follows immediately from Proposition \ref{t3.5} that
the integral is convergent at $t=0$ for $\Re \, s > k-1$. Moreover
the standard method shows that $\delta_\varepsilon(F, r)( s)$ in
(\ref{ps3}) has a meromorphic continuation to the whole complex
plane with simple poles at $s=k-1, k, \ldots$. However the possible
simple pole at $s=0$ is canceled by that of $\Gamma(s)$. Hence $s=0$
is a regular point. From this discussion, we also derive the
following formula
\begin{align} \label{ps4}
\delta_\varepsilon(F)(r) = \delta_\varepsilon(F, r)'( 0) \hspace{2.5in}\\
= -\frac{1}{2}  \int_0^{1} \left( {\rm Tr}_s\left[{ \widehat{
D}^{\pi^*\mu\otimes F}_\varepsilon
 }D^{\pi^*\mu\otimes
F}_\varepsilon(r)e^{-t \left(D^{\pi^*\mu\otimes
F}_\varepsilon(r)\right)^2}\right] -c_{-k}(\var) \, t^{-k} \right)
dt  \nonumber \\
-\frac{1}{2} \int_1^{+\infty} {\rm Tr}_s\left[{ \widehat{
D}^{\pi^*\mu\otimes F}_\varepsilon
 }D^{\pi^*\mu\otimes
F}_\varepsilon(r)e^{-t \left(D^{\pi^*\mu\otimes
F}_\varepsilon(r)\right)^2}\right]dt + C,\hspace{.6in} \nonumber
\end{align}
where $C=c_{\frac{3}{2}}(\var)$ if $m$ is odd (and hence
$k=\frac{3}{2}$) and $C=\frac{1}{2}\Gamma'(1) c_{-1}(\var)$ if $m$
is even (and thus $k=-1$).
\end{proof}

We now define our torsion form. As in the discussion above, we first
define the corresponding zeta function
\begin{align}\label{3.366}  \zeta_{\cal T}(s)=
 -\frac{1}{\Gamma(s)} \int_{0}^{\infty} t^{s-1} \varphi\,
 {\rm Tr}_s\left[N_ZD_t^2e^{\left(1+r^2\right)D_t^2}\right] dt.  \end{align}
From \cite[Theorem 3.21]{BL}, ${\rm Tr}_s
\left[N_Z(1+2D_t^2)e^{\left(1+r^2\right)D_t^2}\right]$ has an
asymptotic expansion as $t\to 0$ with no singular terms (i.e. no
singular powers of $t$). By \cite{BZ} and \cite{dm}, ${\rm Tr}_s
\left[N_Ze^{\left(1+r^2\right)D_t^2}\right]$ has an asymptotic
expansion as $t\to 0$ starting with the $t^{-l}$ term, with $l=0$ if
$n$ is even and $l={1\over 2}$ if $n$ is odd, compare
\cite[(3.118)]{MZ}. Hence ${\rm
Tr}_s\left[N_ZD_t^2e^{\left(1+r^2\right)D_t^2}\right]$ has an
asymptotic expansion as $t\to 0$ starting with the $t^{-l}$ term:
\[ {\rm Tr}_s\left[N_ZD_t^2e^{\left(1+r^2\right)D_t^2}\right] \sim A_{-l} \, t^{-l}
+ A_{-l+1}\, t^{-l+1} + \cdots . \] It follows that $\zeta_{\cal
T}(s)$ has a meromorphic continuation to the whole complex plane
with $s=0$ a regular point. Also, for later use, we note that
\begin{align} \label{nae}
\zeta_{\cal T}(0)=0, \end{align} when $n$ is odd; and \begin{align}
\label{naet} \{ \zeta_{\cal T}(0) \}^{[i]}=0, \end{align} for $i>0$
when $n$ is even.
 Now we define our torsion form by
\begin{align}\label{3.36} {\cal T}_r = \zeta_{\cal T}'(0).
\end{align}
In fact, we have \begin{align} \label{dmz} {\cal T}_r= -
\int_{0}^{1} \varphi\, \left(
 {\rm Tr}_s\left[N_ZD_t^2e^{\left(1+r^2\right)D_t^2}\right] - A_{-l}\, t^{-l} \right) {dt\over
 t} \\
 - \int_{1}^{\infty} \varphi\,
 {\rm Tr}_s\left[N_ZD_t^2e^{\left(1+r^2\right)D_t^2}\right] {dt\over
 t}  + C', \hspace{1.1in} \nonumber
 \end{align}
 where $C'=\varphi\, ( A_{-l} \Gamma'(1))$ if $n$ is even, and $C'=\varphi\, (2 A_{-l})$ if $n$ is odd.

We also introduce a variant of the torsion form. From (\ref{3.46}),
we have for $t\to 0$ \begin{align} \label{dmz2} {\rm
Tr}_s\left[t^{-\frac{1}{2}}D_{4t} \frac{\partial}{\partial \sqrt{t}}
 \left( C_{4t} + \sqrt{-1} r D_{4t} \right)   e^{(1+r^2)D_{4t}^2}
 \right] \sim C_{-k} t^{-k} + C_{-k+1}t^{-k+1} + \cdots .
 \end{align}
Here $k$ is defined as in (\ref{3.46}). Define
\begin{align} \label{dmz3} \zeta_{\widetilde{\mathcal T}}(s)=
-\frac{1}{\Gamma(s)} \int_{0}^{\infty} t^{s} \varphi\,{\rm
Tr}_s\left[t^{-\frac{1}{2}}D_{4t} \frac{\partial}{\partial \sqrt{t}}
 \left( C_{4t} + \sqrt{-1} r D_{4t} \right)   e^{(1+r^2)D_{4t}^2}
 \right] dt. \end{align}
As before, this zeta function has analytic continuation to the whole
complex plane with $s=0$ a regular point. Therefore we can define
\begin{align} \label{dmz4}
\widetilde{\mathcal T}_r=\zeta_{\widetilde{\mathcal T}}'(0).
\end{align}
In fact, one has \begin{align} \label{dmz5} \widetilde{\mathcal
T}_r= - \int_{0}^{1} \varphi\, \left({\rm
Tr}_s\left[t^{-\frac{1}{2}}D_{4t} \frac{\partial}{\partial \sqrt{t}}
 \left( C_{4t} + \sqrt{-1} r D_{4t} \right)   e^{(1+r^2)D_{4t}^2}
 \right]
  - C_{-k}\, t^{-k} \right) dt \\
 - \int_{1}^{\infty} \varphi\,
 {\rm Tr}_s\left[t^{-\frac{1}{2}}D_{4t}
\frac{\partial}{\partial \sqrt{t}}
 \left( C_{4t} + \sqrt{-1} r D_{4t} \right)   e^{(1+r^2)D_{4t}^2}
 \right] dt + C'',
 \hspace{.5in} \nonumber \end{align}
where $C''=2\varphi C_{-\frac{3}{2}}$ if $m$ is odd (and hence
$k=\frac{3}{2}$) and $C''=\Gamma'(1) \varphi C_{-1}$ if $m$ is even
(and thus $k=1$).

We are now ready for our main result.

\begin{thm}\label{mt} Under the assumption that the flat vector bundle $F$ over $M$ is fiberwise acyclic,
  the following identity holds,
\begin{align}\label{3.35}
\lim_{\var \rightarrow 0} \delta_\var(F)(r)= \int_B L\left(TB,
\nabla^{TB}\right)
 \ch\left(\mu, \nabla^{\mu}\right) {\cal T}_r.
\end{align}
\end{thm}
\begin{proof} The proof follows the same line as in the proof of Theorem \ref{t2.4}. Using Proposition \ref{t3.4},
Proposition \ref{t3.5}, (\ref{ps4}) and (\ref{dmz5}), one deduce
that \[  \lim_{\var \rightarrow 0} \delta_\var(F)(r)= {1\over 2}
\int_B L\left(TB, \nabla^{TB}\right)
 \ch\left(\mu, \nabla^{\mu}\right) {\widetilde{\cal T}}_r.\]

To derive the final result, we use
\[ 2u \frac{\partial}{\partial u} C_u=-[N_Z, D_u], \ \ \ \ \ \  2u \frac{\partial}{\partial u} D_u=-[N_Z, C_u] \]
(see \cite[(3.81)]{MZ}) to rearrange the right hand side of
(\ref{3.44}) as in (\ref{2.31}). Namely, one has
\begin{align} \label{dmz6} - {\rm
Tr}_s\left[t^{-\frac{1}{2}}D_{4t} \frac{\partial}{\partial \sqrt{t}}
 \left( C_{4t} + \sqrt{-1} r D_{4t} \right)   e^{(1+r^2)D_{4t}^2}
 \right] = \hspace{1in} \\
 - \frac{2}{t}
 {\rm Tr}_s\left[N_ZD_t^2e^{\left(1+r^2\right)D_t^2}\right] +
 \frac{\sqrt{-1}r}{t} d\,
 {\rm Tr}_s\left[N_ZD_t\, e^{\left(1+r^2\right)D_t^2}\right].
 \nonumber
 \end{align}
It follows that $-2 A_{i} + \sqrt{-1} r \, d\, B_{i} = 0$ if $i <
-\frac{1}{2}$ and $-2 A_{i} + \sqrt{-1} r \, d\, B_{i} = C_{i-1}$ if
$i \geq -\frac{1}{2}$, where $B_i$ is the coefficient of asymptotic
expansion of ${\rm Tr}_s\left[N_ZD_t\,
e^{\left(1+r^2\right)D_t^2}\right]$. Consequently, we obtain by
using (\ref{dmz}) and (\ref{dmz5}),
\[ \int_B L\left(TB, \nabla^{TB}\right)
 \ch\left(\mu, \nabla^{\mu}\right) {\widetilde{\cal T}}_r = 2\int_B L\left(TB, \nabla^{TB}\right)
 \ch\left(\mu, \nabla^{\mu}\right) {\cal T}_r. \]

\end{proof}

\subsection{Comparison with the Bismut-Lott torsion form}\label{s3.5}

$\quad$ Recall that the Bismut-Lott torsion form $\mathcal{T}(T^HM,g^{TZ}, h^F)$
is defined by
\begin{multline}\label{0d000}
\mathcal{T}\left(T^HM,g^{TZ}, h^F\right)=-\varphi
\int_0^{+\infty} \left(\tr_s\left[N_Z\left(1+2D_u^2\right)e^{D_u^2}\right]\right.\\
- d(H(Z,F|_Z)) \left. - \left(
  \frac{n}{2} \chi (Z)\rk (F) - d(H(Z,F|_Z))\right)\left(1-{u\over 2}\right)e^{-u/4}\right){du\over  2u}.
\end{multline}

Now we note that the second and the third terms of the integrand,
terms inserted in (\ref{0d000}) to make the integral convergent,
are degree $0$ terms. Hence,  for $i>0$,
\[  \left\{ \mathcal{T}\left(T^HM,g^{TZ}, h^F\right) \right\}^{[i]} =
 -  \int_0^{+\infty} \left\{ \varphi \, \tr_s\left[N_Z\left(1+2D_u^2\right)e^{D_u^2}\right]\right\}^{[i]} {du\over  2u}    , \]
where we denote by a superscript $[i]$ the $i$-form component of the
corresponding form.

On the other hand, since
\[
\left\{  \tr_s\left[ N_Z D_u^2 \exp \left(D_u^2\right)\right]
\right\} ^{[i]} = u^{-i/2}\left\{  \tr_s\left[ N_Z u D_1^2 \exp
\left(u D_1^2\right)\right] \right\} ^{[i]}, \]
\[
\left\{  \tr_s\left[ N_Z \exp \left(D_u^2\right)\right]
\right\} ^{[i]} = u^{-i/2}\left\{  \tr_s\left[ N_Z \exp
\left(u D_1^2\right)\right] \right\} ^{[i]},
\]
 one deduces that, for $\Re\, s$ sufficiently large,
\begin{multline}
\int_0^{\infty} u^s \left\{  \tr_s\left[ N_Z D_u^2 \exp
\left(D_u^2\right)\right] \right\} ^{[i]} \frac{du}{u}
=\int_0^{\infty} u^{s-\frac{i}{2}} \left\{ \tr_s\left[ N_Z  D_1^2
\exp \left(u
D_1^2\right)\right] \right\} ^{[i]} du \nonumber \\
=\int_0^{\infty} u^{s-\frac{i}{2}}\frac{\partial}{\partial u}
\left\{\tr_s\left[N_Z \exp \left(u D_1^2\right)\right] \right\}^{[i]} du \nonumber \\
=\int_0^{\infty} (i-2s) \, u^{s-\frac{i}{2}} \left\{ \tr_s\left[ N_Z
\exp
\left(u D_1^2\right)\right] \right\} ^{[i]} \frac{du}{2u},  \hspace{-.3in} \nonumber \\
= \int_0^{\infty} (i-2s) u^s\,  \left\{ \tr_s\left[ N_Z \exp \left(
D_u^2\right)\right] \right\} ^{[i]} \frac{du}{2u}, \hspace{1in}
\nonumber
\end{multline}
Cf.  \cite[(3.140) and (3.141)]{MZ}. We have used our assumption
that $H^*(Z_b,F|_{Z_b})=\{0\}$.

Thus,  for $i>0$,
\begin{multline}
\frac{1}{\Gamma(s)} \int_0^{+\infty} u^{s} \left\{  \varphi \,
 \tr_s\left[N_Z\left(1+2D_u^2\right)e^{D_u^2}\right]\right\}^{[i]} {du\over  2u}
 = \\
 (\frac{1}{i-2s} +1)\, \frac{1}{\Gamma(s)} \int_0^{+\infty} u^{s-1} \left\{  \varphi \,
\tr_s\left[ N_Z D_u^2 \exp \left(D_u^2\right)\right] \right\} ^{[i]}
 du
\end{multline}
Hence, \[ \left\{ \mathcal{T}\left(T^HM,g^{TZ}, h^F\right)
\right\}^{[i]} = (1+r^2)^{1-\frac{i}{2}} \left\{ (2+\ln(1+r^2))
\zeta_{{\mathcal T}}(0) + \frac{i+1}{i} \mathcal{T}_r
\right\}^{[i]}, \] as
\[
\left\{ \zeta_{\mathcal{T}}(s) \right\}^{[i]} =- (1+r^2)^{-s +
\frac{i}{2} -1} \frac{1}{\Gamma(s)} \int_0^{\infty} u^{s-1}
 \left\{  \varphi \,
\tr_s\left[ N_Z D_u^2 \exp \left(D_u^2\right)\right] \right\}
^{[i]}\, du.
\]
In particular, using (\ref{nae}) and (\ref{naet}), we have
\begin{align} \label{dmz7} \left\{\mathcal{T}_r\right\}^{[i]}=
\frac{i}{i+1} (1+r^2)^{\frac{i}{2}-1} \left\{
\mathcal{T}\left(T^HM,g^{TZ}, h^F\right) \right\}^{[i]}
\end{align}

For the degree $0$ component, one has
\[  \left \{ {\mathcal T}_r \right\}^{[0]}= 0. \]
This is a direct consequence of \cite[Theorem 3.29]{BL}.
Thus, up to a scaling factor on each degree component,
$\mathcal{T}_r $ captures the positive degree components of the
Bismut-Lott real analytic torsion form.

\end{document}